%% file: bdG-2com.tex
\documentclass[preprint,3p]{elsarticle}

\usepackage{bm,stmaryrd,amsmath,amssymb}
\numberwithin{equation}{section}
\usepackage{tabularx,epsfig,color,tabularx,epstopdf}
\usepackage{epstopdf,caption,graphicx,subcaption,mathtools}
\usepackage{subcaption}
\usepackage{algorithm,algorithmic,multirow}
\usepackage{amsfonts,amsthm}
\usepackage{booktabs}
\usepackage{appendix,color}
\usepackage[colorlinks]{hyperref}
\usepackage{makecell}

\def\<{\left\langle}
\def\>{\right\rangle}

\newtheorem{exam}{Example}
\include{command}

\begin{document}

%\title{An efficient and spectrally accurate numerical method for Bogoliubov-de Gennes excitations in pseudo-spin-1/2 systems}
%\title{On the computation of Bogoliubov-de Gennes excitations for pseudo-spin-1/2 Bose-Einstein condensates}%$ with/without Josephson junction}
\title{Computing the Bogoliubov-de Gennes excitations of two-component Bose-Einstein condensates}%$ with/without Josephson junction}

\author[tju]{Manting XIE}
\ead{mtxie@tju.edu.cn}
\author[tju]{Yong ZHANG\corref{5}}
\ead{Zhang\_Yong@tju.edu.cn}

\cortext[5]{Corresponding author}

\begin{abstract}

In this paper, we present an efficient and spectrally accurate numerical method to compute elementary/collective excitations in two-component Bose-Einstein condensates (BEC), around their mean-field ground state, by solving the associated Bogoliubov-de Gennes (BdG) equation.
The BdG equation is essentially an eigenvalue problem for a non-Hermitian differential operator with an eigenfunction normalization constraint.
Firstly, we investigate its analytical properties, including the exact eigenpairs, generalized nullspace structure
and bi-orthogonality of eigenspaces.
Subsequently, by combining the Fourier spectral method for spatial discretization and
a stable modified Gram-Schmidt bi-orthogonal algorithm,
we propose a structure-preserving iterative method for the resulting large-scale dense non-Hermitian
discrete eigenvalue problem. Our method is matrix-free, and the matrix-vector multiplication (or the operator-function evaluation) is implemented
with a near-optimal complexity ${\mathcal O}(N_{\rm t}\log(N_{\rm t}))$, where $N_{\rm t}$ is the total number of grid points,
thanks to the utilization of the discrete Fast Fourier Transform (FFT).
Therefore, it is memory-friendly, spectrally accurate, and highly efficient.
Finally, we carry out a comprehensive numerical investigation to showcase its superiority in terms of accuracy and efficiency,
alongside some applications to compute the excitation spectrum and Bogoliubov amplitudes in one, two, and three-dimensional problems.
\end{abstract}


%{\bf AMS subject classifications: 65M70, 68Q25, 65T50, 65N25}

\maketitle

\tableofcontents

\section{Introduction}

Early in 1996, right after the first realization of Bose-Einstein Condensations (BECs), researchers created a  two-component BEC using the $|F=2, m=2\rangle$ and $|F=1, m=-1\rangle$ states of $^{87}Rb$ via a double magneto-optic trap \cite{Myatt1997Spin1/2}.
This experiment demonstrated the simultaneous cooling of two condensates and observed inter-component interactions.
Building on this work, it was proposed that binary BEC systems could generate coherent matter waves, or atom lasers, similar to the coherent light emitted by lasers, driving significant interest in multi-component BEC systems.
Subsequent physical experiments have discovered unique quantum phenomena,
such as coherent atom transfer between hyperfine states
and the emergence of density and spin oscillation modes along with a stable equilibrium
\cite{AshhabPRA, Matthews1998Dynamical,Alkai-Japan,Stenger1998Spin, WilliamsPRA}.

At temperature $T$ much smaller than the critical temperature $T_c$ \cite{BaoCai2com,Bao2017Mathematical,Wang-2com}, the two-component BEC is well described by the following coupled Gross-Pitaevskii equations (CGPEs)

\begin{subequations}\label{GPEs}
\begin{align}
i\partial_t \psiu&= \left[-\frac{1}{2} \nabla^2  + V(\bx) + \frac{\delta}{2} + \Big(\betauu |\psiu|^2  + \betaud|\psid|^2\Big)\right]
\psiu + \frac{\Omega}{2}\psid, \label{2comp1} \\
i\partial_t \psid&= \left[-\frac{1}{2} \nabla^2  + V(\bx) - \frac{\delta}{2}+ \Big(\betadu |\psiu|^2  + \betadd|\psid|^2\Big)\right]
\psid + \frac{\Omega}{2}\psiu, \label{2comp2}%\\
\end{align}
\end{subequations}
where $t>0$ denotes time and ${\bx}=x$, ${\bx}=(x, y)^\top$ and ${\bx}=(x, y, z)^\top$ is the Cartesian coordinate vector,
$\Psi(\bx,t) = (\psiu(\bx, t),\psid(\bx, t))^\top$ is the complex-valued macroscopic wave function corresponding to
the spin-up and spin-down components,
$V(\bx)$ is a real-valued  external potential that is case-dependent, and one common choice is the harmonic trapping potential, which reads as
\be\label{Vpoten}
V(\bx) = \frac{1}{2}\left\{\begin{array}{ll}
 \gm_x^2x^2, & d = 1,\\[0.3em]
 \gm_x^2x^2 + \gm_y^2y^2, & d = 2,\\[0.3em]
 \gm_x^2x^2 + \gm_y^2y^2 + \gm_z^2z^2, \ \ &d = 3,
\end{array}\right.
\ee
where $\gm_x$, $\gm_y$ and $\gm_z$ are the trapping frequencies in $x$-, $y$- and $z$-directions, respectively,
$\Omega$ is the effective Rabi frequency to realize the internal atomic Josephson junction by a Raman transition,
$\delta$ is the Raman transition constant, and $\beta_{jl}$ (with $\beta_{jl} = \beta_{lj}$) are the effective interaction
strength between the $j$-th and $l$-th components (positive for repulsive interaction
and negative for attractive interaction).

The GPEs \eqref{GPEs} conserves two important quantities:
the \textsl{mass}
\bea \label{constrain-C}
\mathcal N(\Psi(\cdot,t)):= \|\Psi(\cdot,t)\|^2 = \sum_j  \int_{\mathbb{R}^{d}}|\psi_j(\bx,t)|^{2}\dif\bx  = 1,
\quad t\ge 0,
\eea
and the \textsl{energy}
\bea\label{energy}
\mathcal{E}_\Omega(\Psi(\cdot,t))=
\int_{{\mathbb R}^d} \sum_{j}\frac{1}{2}|\nabla\psi_j|^2 + V(\bx)|\psi_j|^2  +
\sum_{j,l} \frac{\beta_{jl}}{2}|\psi_j|^2 |\psi_l|^2
+ \frac{\delta}{2}\left(|\psiu|^2\! -\! |\psid|^2\right)
+  \Omega\,\mathfrak{Re}(\psiu\overline\psid)\,\dif \bx,
\eea
where $\bar{f}$ and $\mathfrak{Re}(f)$ denotes the conjugate and real part of the complex-valued function $f(\bx)$ respectively.
In fact, we can construct stationary states $ \Phi=(\phi_1,\phi_2)^\top$ of \eqref{GPEs} in the following way
\bea\Psi(\bx,t) =e^{-\im \mu t} ~\Phi(\bx),\eea
where $\Phi$ and $\mu$ (the chemical potential) are solutions to the following Euler-Lagrange equations (a type of nonlinear eigenvalue problem):
\begin{subequations}\label{NLEP}
\begin{align}
\mu\, \phiu  = H_1 \phiu : = \left[-\frac{1}{2} \nabla^2  + V(\bx) + \frac{\delta}{2} + \big(\betauu |\phiu|^2 + \betaud |\phid|^2\big)\right]\phiu + \frac{\Omega}{2}\phid, \\[0.2em]
\mu \,\phid  = H_2\phid : = \left[-\frac{1}{2} \nabla^2  + V(\bx) - \frac{\delta}{2} + \big(\betadu |\phiu|^2 + \betadd |\phid|^2\big)\right]\phid + \frac{\Omega}{2}\phiu,
\end{align}
\end{subequations}
under the constraint
$\mathcal N(\Phi) := \|\Phi\|^2=1$.
The ground state, denoted by $\Phi_g = (\phiug,\phidg)^\top$, is a stationary state with the lowest energy and
can also be defined as the minimizer of the following non-convex optimization problem
\be\label{groundDef}
\Phi_g =\arg\min_{\Phi\in \mathcal S_M} \; \mathcal{E}_\Omega(\Phi),
\ee
where the constraint functional space $\mathcal{S}_M$ reads
$$\mathcal{S}_M:=\Big\{\Phi = (\phiu,\phid)^\top \ \big|\
   \mathcal N(\Phi) =1,~ \mathcal{E}_\Omega(\Phi)<\infty\Big\}.$$

\

For the case without internal Josephson junction ($\Omega = 0$), the mass
of each component is conserved, that is,
\be
\int_{\mathbb{R}^{d}}|\psi_1(\bx,0)|^{2}\dif\bx = \alpha \quad \mbox{ and    }\quad  \int_{\mathbb{R}^{d}}|\psi_2(\bx,0)|^{2}\dif\bx = 1-\alpha, \quad \alpha \in [0, 1].
\ee
The stationary state satisfies the following Euler-Lagrange equations
\begin{subequations}\label{NLEP2}
\begin{align}
   \muu\, \phiu = \left[-\frac{1}{2} \nabla^2  + V(\bx) + \frac{\delta}{2} + \big(\betauu |\phiu|^2 + \betaud |\phid|^2\big)\right]\phiu, \\[0.1em]
\mud\, \phid = \left[-\frac{1}{2} \nabla^2  + V(\bx) - \frac{\delta}{2} + \big(\betadu |\phiu|^2 + \betadd |\phid|^2\big)\right]\phid,
\end{align}
\end{subequations}
under constraints $\|\phiu\|^{2} = \alpha,~\|\phid\|^{2} = 1\!-\!\alpha$ with $\muu$ and $\mud$ being the chemical potentials associated with each component.
The ground state and its corresponding chemical potential, denoted by $\Phi_g^\alpha(\bx) = (\phiua,\phida)^\top$ and ${\muua}$ \& ${\muda}$,
could also be defined as the minimizer of the following problem
\be\label{groundDef2}
\Phi_g^\alpha =\arg\min_{\Phi\in \mathcal S_M^\alpha} \; \mathcal{E}_0(\Phi),
\ee
where $\mathcal S_M^\alpha$ is a nonconvex set defined as
\bea
\mathcal S_M^\alpha:= \left\{\Phi = (\phiu,\phid)^\top\, \Big|\, \|\phi_1\|^2= \alpha,
  ~\|\phi_2\|^2 = 1-\alpha, \, \mathcal{E}_0(\Phi)<\infty   \right\}.
\eea

The GPE \eqref{GPEs} holds for the two-component BECs at the mean-field level \cite{BaoCai2com,Bao2017Mathematical}.
Nonetheless, due to the many-body effects stemming from interatomic interactions,
excitations are observed even in the ground state.
These excitations can be treated as quasi-particles and are named as elementary or collective excitations \cite{Bloch2008Many}.
To properly account for such excitations, one must extend beyond mean-field theory.
The Bogoliubov theory provides a perturbative analysis around the ground state,
ultimately leading to an eigenvalue problem -- the renowned Bogoliubov-de Gennes (BdG) equation \cite{Baillie2017Collective, Deng2020Spin, Gao2020Numerical, Huhtamaeki2011Elementary}.

Since the first observations of excitation modes \cite{Jin1996Collective},
there has been a significant growing interest in mathematical and numerical studies devoted to the Bogoliubov-de Gennes (BdG) equations over the past few decades \cite{Boccato2020Acta,Hu2004Analytical,Stringari1996Collective}.
Along the numerical front, various research utilized the standard ARPACK library \cite{ARPACK} to solve the BdG equations for different types of BECs.
For an incomplete list of references, see \cite{Edwards1996Collective, Danaila2016Vector, Tang2022Spectrally, Deng2020Spin}.
For problems in lower dimensions ($d=1,2$), the resulting eigenvalue problems are typically solved using direct eigenvalue solvers,
such as MATLAB's ``eigs'' function, or iterative subspace solvers like those provided by ARPACK,
or locally optimal block-preconditioned conjugate gradient methods \cite{Bai2012Minimization, Bai2013Minimization}.

To date, there are quite few mathematical studies, and most numerical investigations have concentrated on lower-dimensional scenarios ($d=1$ or $2$), for example, the low-lying elementary (or collective) excitations in quasi-two-dimensional rotating systems \cite{Deng2020Spin, Zhang2021}.
Chen et al. \cite{Chen2017Collective} employed the Arnoldi method to directly solve the BdG equations for spin-orbit coupled BECs,
and Sadaka et al. \cite{BdG-FreeFEM} developed a numerical eigenvalue solver with FreeFEM and ARPACK.
More recently, a series of spectrally accurate eigenvalue solvers were proposed
for dipolar/spin-1 BECs based on ARPACK or Bi-Orthogonal Structure-Preserving solver (BOSP) \cite{LiWangZhang} in high space dimension \cite{LLXZ-BdGSpin1,Tang2022Spectrally,Zhang2021}.

So far as we know, when solving BdG equations numerically, two major challenges arise:
(i) the development of an accurate stationary state solver;
(ii) the design of an accurate and efficient eigensolver for the BdG equation.
For the stationary/ground state computation, several numerical methods have been proposed,
including the gradient flow method \cite{Bao2017Mathematical, Bao2007mass,Liu2021Normalized},
regularized Newton method \cite{TianCai} and Preconditioned Conjugate Gradient (PCG) method \cite{ATZ-CiCP-PCG} etc.
Here, in this article, we choose PCG \cite{ATZ-CiCP-PCG} for its excellent performance in accuracy,
efficiency and robustness.
 While the BdG equation generally admits eigenvalue zero, and, for current popular eigensolvers, the most challenging tasks consist of the
convergence at eigenvalue zero and the overall computational efficiency that usually bottlenecks numerical simulation, especially in high spatial dimensions.

It is well known that the eigenvalue distribution and structure of the generalized nullspace are of essential importance
to the design of eigensolvers to achieve better performance.
For the BdG equation, all the eigenvalues come in negative/positive pairs along the real line,
and the eigenspaces associated with eigenvalues of different magnitudes are biorthogonal to each other \cite{LiWangZhang}.
While for eigenvalue zero, the algebraic multiplicity is usually greater than its geometric counterpart,
therefore, the generalized nullspace has a very complicated structure.
Yet, not all such intrinsic properties and structures have been adequately taken into account in current solvers.
Simple adaptation of existing solvers, for example, ARPACK \cite{ARPACK} and LOBP4DCG \cite{Bai2012Minimization,Bai2013Minimization},
might suffer from slow convergence or even divergence, especially when there is an eigenvalue of zero,
and the performance is much more likely to degrade significantly in large-scale computation.
To make it worse, the eigenfunctions are typically discretized with the Fourier spectral method, leading to a fully populated dense eigensystem, and explicit matrix storage incurs prohibitive memory costs, especially for high-dimensional problems.
Therefore, it is imperative to develop an iterative algorithm where the user can provide a matrix-vector product implementation,
waiving any explicit matrix storage, thus enabling computation of large-scale problems in high space dimensions.

Recently, the BOSP algorithm, proposed by Li et al \cite{LiWangZhang}, successfully incorporated the eigenspaces' biorthogonality
using a modified Gram-Schmidt biorthogonalization algorithm,
achieving great efficiency and parallel scalability. Moreover, it provided an interactive interface through which the users are allowed to implement their matrix-vector product.
It shall perform better if the generalized nullspaces could be given \textit{a priori} whenever it is possible.
In this article, to adapt BOSP for our problem, we shall first investigate the generalized nullspace structure on the analytical level, then
design an efficient matrix-vector product, which should be an accurate approximation of operator-function evaluation, aiming
to achieve spectral-accuracy computation of all non-zero eigenvalues and their corresponding eigenfunctions.

\

Overall, the main objectives of this paper are as follows:
\begin{itemize}
\item[(1)] derive the Bogoliubov-de Gennes (BdG) equation for  two-component BECs around the ground state and investigate its mathematical structures, including analytical eigenvalues/eigenfunctions, the generalized nullspace, and the bi-orthogonality of eigenspaces;

\item[(2)] develop a bi-orthogonal structure-preserving Fourier spectral method and propose an effective implementation for evaluating the matrix-vector product based on the Fourier spectral method, especially FFT;

\item[(3)] verify the spectral accuracy, efficiency, and scalability of the method, and numerically study the excitation spectrum and Bogoliubov amplitudes around the ground state for different parameters in one-dimensional (1D) -- three-dimensional (3D) settings.
\end{itemize}

The rest of the paper is organized as follows:
In Section \ref{sec:BdG_Prop}, we introduce the BdG equations and derive some analytical properties.
In Section \ref{Numer_alg}, we present details of the Fourier spectral method for space discretization
and propose an efficient numerical method.
Extensive numerical examples are shown in Section \ref{NumResult} to confirm the performance of our method,
together with some applications to study the solutions to the BdG equations with different parameters in 1D--3D.
Finally, conclusions are drawn in Section \ref{sec:Conclusion}.

%=================================================
\section{The BdG equation and its properties}
\label{sec:BdG_Prop}

\subsection{The Bogoliubov-de Gennes equations}
\label{sec:BdG equation}
To characterize the elementary/collective excitations of a two-component BEC,
the Bogoliubov theory \cite{Baillie2017Collective,Deng2020Spin}
begins with the stationary state $\Phi$ of the CGPEs \eqref{GPEs} with an internal atomic Josephson junction ($\Omega\neq 0$),
which is also an solution to the nonlinear eigenvalue problems \eqref{NLEP} with corresponding chemical potential $\mu$,
and assumes the evolution is around state $\Phi$.
The corresponding wave function $\Psi$ takes the following form \cite{YiLowLying18,Tang2022Spectrally}
\bea\label{waveAssump} \quad
\Psi(\bx,t) = \exp({-\im \mu t})
\left[ \Phi(\bx) + \varepsilon \sum_{\ell=1}^\infty \Big( \bu^\ell(\bx) e^{-\im \omega_\ell t} + \bar\bv^\ell (\bx) e^{\im \omega_\ell t}\Big)\right],
\ \bx\in\mathbb{R}^d,\ t > 0.\ \
\eea
Here, $0<\varepsilon \ll 1$ is a small quantity used to control the population of quasiparticle excitation, $\mathbf{\omega_\ell\in \mathbb{R}}$
is the frequency of the excitations to be determined,
and $\bu^\ell, \bv^\ell \in H^1(\mathbb R^d;\mathbb C^2)$ are the Bogoliubov excitation modes
satisfying normalization condition
\bea
\label{constrain_org}
\int_{\mathbb R^d} \left( | \bu^\ell(\bx)|^2-  | \bv^\ell(\bx)|^2 \right) \dif \bx =
\int_{\mathbb R^d} \sum_{j}\left( | u^\ell_j(\bx)|^2-  | v^\ell_j(\bx)|^2 \right) \dif \bx =  1, \quad \ell\in \mathbb Z^{+}.
\eea
where $\bu^\ell = (\uul,\udl)^\top,\bv^\ell =(\vul,\vdl)^\top$.
Plugging \eqref{waveAssump} into \eqref{GPEs}, by collecting linear terms in $\varepsilon$ and separating frequency $e^{-\im\omega_\ell t}$ and $e^{\im\omega_\ell t}$,
we obtain the BdG equations as follows
\begin{gather}
\label{BdG_uv}
\mathcal L
\begin{bmatrix}
\bu \\
\bv
\end{bmatrix}
:=
\begin{bmatrix}
\cA & \cB \\
\cC  & \cD
\end{bmatrix}
\begin{bmatrix}
\bu \\
 \bv
\end{bmatrix}
=\omega\begin{bmatrix}
\bu\\
\bv
\end{bmatrix},
\end{gather}
with constraint
\be
\label{constrain}
\int_{\mathbb{R}^d} \left(|\bu(\bx)|^2-|\bv(\bx)|^2\right)\,\dif\bx=1,
\ee
where all the subscripts $\ell$ are omitted hereafter for simplicity, $\omega$ is the excitation energy and the operators are given explicitly
\bea
\cA := \begin{bmatrix}
   L_1 & \betaud\phiu\overline{\phid} + \frac{\Omega}{2} \\[0.4em]
\betadu\overline{\phiu}{\phid} + \frac{\Omega}{2} & L_2
\end{bmatrix},& &
\cB := \begin{bmatrix}
\betauu(\phiu)^2 & \betaud{\phiu}{\phid} \\[0.4em]
\betadu{\phiu}{\phid} & \betadd(\phid)^2
\end{bmatrix}, \label{def_B}\\
\cC:= -\begin{bmatrix}
\betauu(\overline\phiu)^2 & \betaud\overline{\phiu}\overline{\phid} \\[0.4em]
\betadu\overline{\phiu}\overline{\phid} & \betadd(\overline\phid)^2
\end{bmatrix}, & &
\cD :=- \begin{bmatrix}
L_1 & \betaud\overline{\phiu}{\phid} + \frac{\Omega}{2} \\[0.4em]
\betadu{\phiu}\overline{\phid} + \frac{\Omega}{2} & L_2
\end{bmatrix}, \label{def_D}
\eea
with
\beas
L_1 &=& -\frac{1}{2} \nabla^2  + V(\bx) + \frac{\delta}{2} + 2\betauu |\phiu|^2 + \, \betaud |\phid|^2 - \mu_g,\\
L_2 &=& -\frac{1}{2} \nabla^2  + V(\bx) - \frac{\delta}{2} + \, \betadu |\phiu|^2 + 2\betadd |\phid|^2 - \mu_g.
\eeas

It is easy to check that $\mathcal A$ and $\mathcal D$ are both Hermitian operators,
i.e., $\cA^* = \cA, ~\cD^*=\cD$, and $\cB^*= -\cC$, where symbol $*$ denotes
the adjoint operator  associated with inner product $\left\langle \bbf,\bbg \right\rangle
:=\sum_{j=1,2}\int_{\mathbb{R}^d} {f_j}(\bx)\overline{g_j(\bx)}\, \dif \bx$.
Furthermore, we have $\overline{\cA\bbg} = -\cD\bar\bbg$,
$\overline{\cB\bbg} = -\cC\bar\bbg$, $\forall\,\bbg\in H^1(\mathbb{R}^d;\mathbb C^2)$,
and it immediately implies the operators' finite-dimensional subspace representations,
denoted by matrices $A,B,C$ and $D$, are either Hermitian or symmetric, that is,
\bea
A^\sH = A, \quad D = -\overline A ,  \quad B^\top = B \quad \mbox{and} \quad C  = -\overline B.
\eea
The matrices representation of \eqref{BdG_uv} reads as follows
\begin{gather}
\begin{bmatrix}
A & B \\
-\overline B   &- \overline A
\end{bmatrix}
\begin{bmatrix}
{\mathbf u} \\
{\mathbf v}
\end{bmatrix}
=\omega\begin{bmatrix}
{\mathbf u} \\
{\mathbf v}
\end{bmatrix},
\end{gather}
and it coincides with the Bethe-Salpeter Hamiltonian (BSH) matrix arising from optical absorption spectrum analysis \cite{Shao2018structure}.

\

\

Similarly, if there is no internal atomic Josephson junction (i.e. $\Omega=0$) in \eqref{GPEs}, for any given
$\alpha \in [0, 1]$, the corresponding wave function $\Psi$ takes the following form \cite{YiLowLying18,Tang2022Spectrally}
\bea\label{waveAssump1} \quad
\Psi(\bx,t) =\diag\big(\exp({-\im \muua t}),\exp({-\im \muda t})\big)
\left[ \Phi_g^\alpha(\bx) + \varepsilon \sum_{\ell=1}^\infty \Big( \bu^\ell(\bx) e^{-\im \omega_\ell t} + \bar\bv^\ell (\bx) e^{\im \omega_\ell t}\Big)\right],\nonumber\\
\ \bx\in\mathbb{R}^d,\ t > 0,
\eea
and obtain the following similar BdGEs:
\begin{gather}
\label{BdG_uv-woJJ}
\mathcal L_\alpha
\begin{bmatrix}
\bu \\
\bv
\end{bmatrix}
:=
\begin{bmatrix}
\cA_\alpha & \cB_\alpha \\
\cC_\alpha  & \cD_\alpha
\end{bmatrix}
\begin{bmatrix}
\bu \\
\bv
\end{bmatrix}
=\omega\begin{bmatrix}
\bu\\
\bv
\end{bmatrix},
\end{gather}
with constraint \eqref{constrain},
where $\omega$ is the excitation energy corresponding to the excitation
mode $(\bu^\top,\bv^\top)^\top$ and and $\cA_\alpha$, $\cB_\alpha$, $\cC_\alpha$ and $\cD_\alpha$ are defined as
\beas
&\cA_\alpha = \begin{bmatrix}
L_1^\alpha & \betaud\phiua\overline{\phida}\\
\betadu\overline{\phiua}{\phida} & L_2^\alpha
\end{bmatrix},\quad
\cB_\alpha = \begin{bmatrix}
\betauu(\phiua)^2 & \betaud{\phiua}{\phida} \\
\betadu{\phiua}{\phida} & \betadd(\phida)^2
\end{bmatrix},\label{def_B1}\\
&\cC_\alpha = -\begin{bmatrix}
\betauu(\overline\phiua)^2 & \betaud\overline{\phiua}\overline{\phida} \\
\betadu\overline{\phiua}\overline{\phida} & \betadd(\overline\phida)^2
\end{bmatrix},\quad
\cD_\alpha = -\begin{bmatrix}
L_1^\alpha & \betaud\overline{\phiua}{\phida}\\
\betadu{\phiua}\overline{\phida}& L_2^\alpha
\end{bmatrix},\label{def_D1}
\eeas
with
\beas
L_1^\alpha = -\frac{1}{2} \nabla^2  + V(\bx) + \frac{\delta}{2} + 2\betauu |\phiua|^2 + \, \betaud |\phida|^2 - \muua,\\
L_2^\alpha = -\frac{1}{2} \nabla^2  + V(\bx) - \frac{\delta}{2} + \, \betadu |\phiua|^2 + 2\betadd |\phida|^2 - \muda.
\eeas

\subsection{Analytical properties}
\label{sec:Prop}
In this subsection, we derive some analytical properties of the BdG equations as well as the structure of eigenfunctions for the two-component BEC with internal atomic Josephson junction ($\Omega\neq 0$), which might serve as benchmarks for numerical solutions or help to design an efficient numerical method for eigenvalue problems.

Here, it should be noted that although the establishment of the model requires $\omega$ to be a real number, for the following model eigenvalue problem
\begin{gather}
\label{BdG_uv1}
\begin{bmatrix}
\cA & \cB \\
\cC  & \cD
\end{bmatrix}
\begin{bmatrix}
\bu \\
 \bv
\end{bmatrix}
=\lambda\begin{bmatrix}
\bu\\
\bv
\end{bmatrix},
\end{gather}
its eigenvalue $\lambda$ may be complex. Thus, we have the following theorem:

\begin{thm}[\textbf{Symmetric distribution}]\label{lem:negtive_eigenvalue}
If $\{\lambda; \bu, \bv\} (\lambda\in \mathbb C)$ is a solution pair to the eigenvalue problem \eqref{BdG_uv1},
then $\{-{\bar\lambda};\bar{ \bv}, \bar{ \bu} \}$ is also a solution pair.
Furthermore,  if $(\bu,\bv)$ satisfies the normalization constraint \eqref{constrain}, the corresponding eigenvalue is real, that is, $\omega\in \mathbb R$.

\end{thm}
\begin{proof} The first conclusion can be proved by taking the conjugate of Eqn. \eqref{BdG_uv}.
Multiplying the first/second equation of \eqref{BdG_uv} by $\bu$/$\bv$ respectively
and integrating each equation concerning $\bx$,
we subtract the second integration from the first one to obtain the following
\begin{equation}\label{real_eg3}
\lag \cA \bu, \bu\rag + \lag \cB \bv,\bu \rag - \lag\cC \bu,\bv\rag
- \lag\cD \bv,\bv\rag = \omega\int_{\mathbb{R}^d}(|\bu(\bx)|^2-|\bv(\bx)|^2)\,\dif \bx.
\end{equation}
From Eqn. \eqref{constrain}, we can see that $\omega$ is real.
\end{proof}

\begin{thm}[\textbf{Analytical eigenparis}]\label{lemAnalytical}
   Let $\Phi=(\phiu,\phid)^\top$ be a stationary state for Eqn. \eqref{GPEs} with harmonic trapping potential \eqref{Vpoten},
we have analytical solutions to the BdG equation \eqref{BdG_uv} as follows
\bea\label{eq:eig-pair}
\{\og_{\sigma};\bu_{\sigma},\bv_{\sigma}\}:=\left\{\gm_{\sigma} ; \fl{1}{\sqrt{2}} \lf \gm_{\sigma}^{-\fl{1}{2}}
\p_{\sigma} \Phi-\gm_{\sigma}^{\fl{1}{2}}{\sigma}\Phi \rg, \fl{1}{\sqrt{2}} \lf \gm_{\sigma}^{-\fl{1}{2}}\p_{\sigma}\overline\Phi+\gm_\sigma^{\fl{1}{2}}{\sigma}\overline\Phi \rg \right\},
\eea
with ${\sigma} = x$ in one dimension, ${\sigma} = x,y$ in two dimensions and ${\sigma}= x, y, z$ in three dimensions.
\end{thm}
\begin{proof}
For simplicity, we only prove the ${\sigma} =x$ case, and extensions to other spatial variables are similar.
Differentiate \eqref{NLEP} with respect to $x$, we derive
\bea
L_1(\partial_x\phiu)  + \gamma_x^2x\phiu + \betauu(\phiu)^2(\partial_x\overline{\phiu}) + \betaud\Big[\phiu\overline{\phid}(\partial_x\phid)+\phiu{\phid}(\partial_x\overline{\phid})\Big] + \frac{\Omega}{2}\partial_x\phid&=& 0,\label{ev1}\\[0.2em]
L_2(\partial_x\phid)  + \gamma_x^2x\phid +  \betadu\Big[\overline{\phiu}{\phid}(\partial_x\phid)+\phiu{\phid}(\partial_x\overline{\phiu})\Big] + \betadd(\phid)^2(\partial_x\overline{\phid}) + \frac{\Omega}{2}\partial_x\phiu&=& 0. \label{ev2}
\eea
Multiply both sides of \eqref{NLEP} by $\gamma_xx$, we have
\bea
L_1(\gamma_xx\phiu)  -\betauu|\phiu|^2 (\gamma_xx\phiu) + \gamma_x(\partial_x\phiu) + \frac{\Omega}{2}\gamma_xx\phid&=& 0, \label{ev3}\\
L_2(\gamma_xx\phid)  -\betadd|\phid|^2 (\gamma_xx\phid) + \gamma_x(\partial_x\phid) + \frac{\Omega}{2}\gamma_xx\phiu&=& 0. \label{ev4}
\eea
By subtracting Eqn. \eqref{ev3} from Eqn. \eqref{ev1} and Eqn. \eqref{ev4} from Eqn. \eqref{ev2}, we get
\be\label{aeq1}
\cA(\partial_x\Phi-\gamma_xx\Phi) + \cB(\partial_x\overline\Phi+\gamma_xx\overline\Phi) = \gamma_x (\partial_x\Phi-\gamma_xx\Phi).
\ee
Summing Eqn. \eqref{ev1}\&\eqref{ev3} and Eqn. \eqref{ev2}\&\eqref{ev4}, and taking complex conjugates, we obtain
\be\label{aeq2}
\cC(\partial_x\Phi-\gamma_xx\Phi) + \cD(\partial_x\overline\Phi+\gamma_xx\overline\Phi) = \gamma_x (\partial_x\overline\Phi+\gamma_xx\overline\Phi).
\ee
Therefore, we can verify that $\{\partial_x\Phi-\gamma_xx\Phi, \partial_x\overline\Phi+\gamma_xx\overline\Phi\}$
solves \eqref{HpHm-eq} with $\omega =\gamma_x$.

Simple calculations lead to the following identity
\bes
\int_{\mathbb{R}^d} \left(|\partial_x\Phi-\gamma_xx\Phi|^2-|\partial_x\overline\Phi+\gamma_xx\overline\Phi|^2\right)\,\dif\bx=-2\int_{\mathbb{R}^d}\left( \gamma_xx\Phi\partial_x\overline\Phi + \gamma_xx\overline\Phi\partial_x\Phi  \right)\,\dif\bx = 2{\gamma_x},
\ees
then we derive the analytic solutions as \beas
\og_x=\gm_x, \quad \bu_x=\fl{1}{\sqrt{2}}\Big(\gm_x^{-\fl{1}{2}}\p_x\Phi-\gm_x^{\fl{1}{2}}x\Phi\Big), \quad
\bv_x=\fl{1}{\sqrt{2}}\Big(\gm_x^{-\fl{1}{2}}\p_x\overline\Phi+\gm_x^{\fl{1}{2}}x\overline\Phi\Big).\eeas
\end{proof}

As shown in \cite{Bao2017Mathematical}, under appropriate assumptions,  the ground state $\Phi_g=(\phiug,\phidg)^\top$ are \textbf{real-valued} functions.
Therefore, we shall focus on the real ground state hereafter, and all operators involved are real.
To be more specific,
\bes\label{real-oper}
\cD = -\cA,\quad \cC = -\cB,
\ees
then the BdG equation \eqref{BdG_uv} becomes a linear response problem of the following form
\begin{gather}
\label{BdG_new}
\begin{bmatrix}
\cA & \cB \\
-\cB  & -\cA
\end{bmatrix}
\begin{bmatrix}
\bu\\
\bv
\end{bmatrix}
=\omega\begin{bmatrix}
\bu\\
\bv
\end{bmatrix}.
\end{gather}

By applying a change of variables $\bu=\bbf+\bbg,~\bv=\bbf-\bbg$, the above equation can be reformulated
\bea \label{HpHm-eq}
\cH\begin{bmatrix}
	\bbf  \\ \bbg
	\end{bmatrix} :=
\begin{bmatrix}
     \mathcal O&\mathcal H_- \\
     	\mathcal H_+&\mathcal O
\end{bmatrix}
\begin{bmatrix}
	\bbf  \\ \bbg
	\end{bmatrix} =
	\omega
\begin{bmatrix}
	 \bbf  \\ \bbg
\end{bmatrix},
\eea
where $\mathcal H_+:=\cA+\cB$ and $\mathcal H_-:=\cA-\cB$ are both Hermitian operators, and the constraint \eqref{constrain} is reformulated as
$\lag\bbf,{\bbg}\rag=1/4$.
Equation \eqref{HpHm-eq} immediately leads to two decoupled linear eigenvalue problems, that is,
\bea
\mathcal H_-\mathcal H_+\bbf=\og^2\bbf, \quad \quad
\mathcal H_+\mathcal H_-\bbg=\og^2\bbg. \label{de-eigen}
\eea

%\begin{remark}
%If $\cH_+$ is positive definite, according to Lemma 2.3 of \cite{LiWangZhang},
%the constrain \eqref{constrain} implies that the eigenvalue $\omega$ is nonzero, therefore,
%we shall only focus on the {\bf non-zero} eigenvalues and their eigenfunctions.
%\end{remark}

\begin{thm}[\textbf{Biorthogonality}]\label{lem:orth}
Assume $\{\omega_i; \bbf_i, \bbg_i\}_{i=1}^2$ are eigenpairs of Eqn. \eqref{HpHm-eq}
with eigenvalues of different magnitudes, i.e., $|\omega_1|\neq|\omega_2|$,
the following biorthogonal properties hold true
\begin{equation*}
\langle \bbf_1,\bbg_2 \rangle
=
\langle \bbf_2,\bbg_1 \rangle
=0.
\end{equation*}
\end{thm}

\begin{proof}
Using Eqn. \eqref{de-eigen}, we have
\bes
\og_i^2\lag\bbf_i,\bbg_j\rag  = \lag\mathcal H_-\mathcal H_+\bbf_i,\bbg_j\rag = \lag\bbf_i,\mathcal H_+\mathcal H_-\bbg_j\rag =  \og_j^2\lag\bbf_i,\bbg_j\rag,
\ees
which means
\bes
(\og_i^2 - \og_j^2)\lag\bbf_i,\bbg_j\rag = 0.
\ees
Since $|\og_i|\neq|\og_j|$, we obtain $\lag\bbf_i,\bbg_j\rag = 0, \ \  \mbox{for}\ i\neq j$.
\end{proof}

\begin{remark}
In this subsection, we only focus on the case with an internal atomic Josephson junction ($\Omega\neq 0$).
Actually, Theorems \ref{lem:negtive_eigenvalue}-\ref{lem:orth} also hold for the case without an internal atomic Josephson junction ($\Omega=0$),
and they are omitted here for brevity.
\end{remark}

\subsection{Generalized nullspace}
\label{sec:rootspace}
As is known, the algebraic/geometric multiplicity of eigenvalue zero and the structure of its associated generalized nullspace
are of great importance to the solver's performance in terms of convergence, accuracy and efficiency.
The generalized nullspace of $\mathcal H$ is defined as nullspace of $\mathcal H^{q}$ for some positive integer $p$ such that
\begin{equation*}
\ns(\cH^p)=\ns(\cH^{p+1}).
\end{equation*}
For the operator $\cH$, the integer $p$ is finite and $\ns(\cH^p)=\ns(\cH^{q})$ holds true for any greater integer $q$, i.e, $q \geq p$.
In this subsection, we shall elaborate on the structure of generalized nullspace, which is, of course closely connected with the nullspace of $\ mathcalH H_-$ and $\mathcal H_+$.

\vspace{0.35cm}

%\noindent {\bf Case I: With an internal atomic Josephson junction ($\Omega\neq 0$)}.

\subsubsection{With an internal atomic Josephson junction ($\Omega\neq 0$)}
%\noindent {\bf Case I: With an internal atomic Josephson junction ($\Omega\neq 0$)}.

First and foremost, we shall present the conditions for the existence and uniqueness of the ground state solution of \eqref{groundDef}.

\begin{lem} [\textbf{Existence and uniqueness of the ground state} \cite{BaoCai2com}]\label{lem-eu-wjj} Suppose $V(x) \geq 0$ satisfying $\lim_{|x|\rightarrow \infty}V(x) =\infty$ and at least one of the following conditions holds,\\
	\indent \ \ (i) $d = 1$;\\
	\indent \ (ii) $d = 2$ and $\betauu \geq C_b, \betadd \geq -C_b$, and $\betaud=\betadu\geq -C_b-\sqrt{C_b +\betauu}\sqrt{C_b +\betadd}$;\\
	\indent (iii) $d = 3$ and $\begin{bsmallmatrix}\betauu& \betaud\\\betadu &\betadd\end{bsmallmatrix}$ is either positive semi-definite or nonnegative,\\
	there exists a ground state $\Phi_g = (\phiug,\phidg)^\top$ of \eqref{groundDef}. Furthermore, if the matrix $\begin{bsmallmatrix}\betauu& \betaud\\\betadu &\betadd\end{bsmallmatrix}$ is positive
	semi-definite and at least one of the parameters $\delta$, $\Omega$, $\betauu-\betadd$ and $\betauu-\betaud$ are nonzero, then the ground
	state $(|\phiug|,-\operatorname{sign}(\Omega)|\phidg|)^\top$ is unique.
\end{lem}

\begin{thm}[\textbf{Nullspace of $\cH_-$ for system with JJ ($\Omega\neq0$)}]\label{TCC}
	Under conditions of Lemma  \ref{lem-eu-wjj}, we have the following property for nullspace of $\mathcal H_-$
	$$\spa \{\Phi_g\}\subset \ns(\cH_-).$$
\end{thm}

\begin{proof}
	Just verify $\cH_-\Phi_g=0$ directly and omit the detailed proof here.
\end{proof}

\begin{remark}\label{rem:withJJ}
	From our extensive numerical results, not fully shown here,
	we conjecture that the following property holds for $\Omega\neq0$:
	\begin{center}
		$\cH_+$ is positive definite ~ ~ ~ and ~ ~ ~ $\ns(\cH_-) = \spa \{\Phi_g\}$.
	\end{center}
\end{remark}

Finally, we present the structure of the generalized nullspace of the operator $\mathcal{H}$.

\begin{thm}[\textbf{Generalized nullspace of $\cH$}]\label{thm:withJJ}
	If the above property is true,
	we can obtain the generalized nullspace of $\mathcal H$
	as follows
	\bea
	\ns(\cH^3) = \ns(\cH^2) = \ns(\cH) \oplus \spa \left\{\begin{bsmallmatrix}
		\widehat\Phi_1\\
		\mathbf 0
	\end{bsmallmatrix}\right\},
	\eea
	where $\widehat{\Phi}_1 = \cH_+^{-1}\Phi_g$, and $\ns(\cH) = \spa \left\{\begin{bsmallmatrix}
			\mathbf 0 \\
			\Phi_g
		\end{bsmallmatrix}\right\}.
	$
\end{thm}
\begin{proof}
	It is easy to verify that
$$\ns(\cH) =\{(\bm f,\bm g)^\top|\cH_-\bm g=\mathbf 0, \cH_+\bm f=0\} = \{(\mathbf 0,\bm g)^\top|\bm g\in\ns(\cH_-)\} = \spa \left\{\begin{bsmallmatrix}
			\mathbf 0 \\
			\Phi_g
		\end{bsmallmatrix}\right\}.
	$$
For $\ns(\cH^2)$ and $\ns(\cH^3)$, we can obtain
\beas\ns(\cH^2) &=&\{(\bm f,\bm g)^\top|\cH_+\cH_-\bm g=\mathbf 0, \cH_-\cH_+\bm f=0\} \\[0.2em]
 &=& \{(\bbf,\bbg)^\top|\cH_-\bm g=\mathbf 0,\cH_+\bbf\in\ns(\cH_-)\}\\[0.3em]
 &=& \{(\bbf,\bbg)^\top|\bm g\in \ns(\cH_-),\bbf\in\cH_+^{-1}\ns(\cH_-)\} \\[0.2em]
 &=&\ns(\cH) \oplus \spa \left\{\begin{bsmallmatrix}
		\widehat\Phi_1\\
		\mathbf 0
	\end{bsmallmatrix}\right\},
\eeas
and
\beas\ns(\cH^3) &=&\{(\bm f,\bm g)^\top|\cH_-\cH_+\cH_-\bm g=\mathbf 0, \cH_+\cH_-\cH_+\bm f=0\} \\[0.3em]
&=& \{(\bbf,\bbg)^\top|\cH_-\bm g\in \cH_+^{-1}\ns(\cH_-),\bbf\in\cH_+^{-1}\ns(\cH_-)\}.
\eeas
Obviously, $\ns(\cH^2)\subset \ns(\cH^3)$. Next, we prove that $\ns(\cH^3)\subset \ns(\cH^2)$.
For any $(\bbf,\bbg)^\top\in \ns(\cH^3)$, i.e., $\cH_-\bm g\in \cH_+^{-1}\ns(\cH_-)$, we have $\cH_-\bm g = c\,\widehat\Phi_1$ and
\begin{equation*}
	0=\langle \mathbf{0}, \bbg \rangle = \langle \mathcal H_-\Phi_g, \bbg \rangle
	=\langle \Phi_g,\mathcal H_- \bbg \rangle
	=\langle \cH_+\widehat\Phi_1,c\,\widehat\Phi_1 \rangle=c\langle \cH_+\widehat\Phi_1,\widehat\Phi_1 \rangle.
\end{equation*}
By the positivity of $\cH_{+}$, we have $c=0$ or $\widehat\Phi_1=\mathbf 0$, i.e., $\mathcal H_- \bbg=\mathbf 0$,
which implies immediately that $\ns(\cH^3)\subset \ns(\cH^2)$. Therefore, $\ns(\cH^3)=\ns(\cH^2)$.
\end{proof}

%\vspace{0.1cm}

%\noindent {\bf Case II: Without an internal atomic Josephson junction ($\Omega = 0$)}.

\subsubsection{Without an internal atomic Josephson junction ($\Omega = 0$)}
%\noindent {\bf Case II: Without an internal atomic Josephson junction ($\Omega = 0$)}.

\vspace{0.5em}
Similar to the preceding scenario, we shall initially delineate the conditions that ensure the existence and uniqueness of the ground state solution of \eqref{groundDef}.

%Denote $\betauu' :=\alpha \betauu, \betadd'=(1-\alpha)\betadd, \betaud'=\betadu' =\sqrt{\alpha(1-\alpha)}\betaud, \alpha' =\alpha(1-\alpha)$. Then the following conclusions can be drawn.
\begin{lem} [\textbf{Existence and uniqueness of the ground state} \cite{BaoCai2com}]\label{lem-eu-wojj} Suppose $V(x) \geq 0$ satisfying $\lim_{|x|\rightarrow \infty}V(x) =\infty$ and at least one of the following conditions holds,\\
	\indent \ \ (i) $d = 1$;\\
	\indent \ (ii) $d = 2$ and $\betauu \geq -C_b/\alpha, \betadd \geq -C_b/(1-\alpha)$, and $\betaud=\betadu\geq -\sqrt{C_b +\alpha\betauu}\sqrt{C_b +(1-\alpha)\betadd}$;\\
	\indent (iii) $d = 3$ and $\begin{bsmallmatrix}\betauu& \betaud\\\betadu &\betadd\end{bsmallmatrix}$ is either positive semi-definite or nonnegative,\\
	there exists a ground state $\Phi_g = (\phiug,\phidg)^\top$ of \eqref{groundDef}. Furthermore, if the matrix $\begin{bsmallmatrix}\betauu& \betaud\\\betadu &\betadd\end{bsmallmatrix}$ is positive
	semi-definite, then the ground
	state $(|\phiug|,|\phidg|)^\top$ is unique.
\end{lem}

\begin{thm}[\textbf{Nullspace of $\cH_-$ for system without JJ ($\Omega=0$)}]\label{TCC}
	Under conditions of Lemma \ref{lem-eu-wojj}, we have the following property for nullspace of $\mathcal H_-$
	$$\spa \{\Phi_1,\Phi_2\}\subset \ns(\cH_-),$$
	where $\Phi_1 = \Phi_g$, $\Phi_2 = (c_{1}\phiug,c_{2}\phidg)^\top$, and
	$(c_{1},c_{2}) = \big(-(\frac{1-\alpha}{\alpha})^{1/2}, (\frac{\alpha}{1-\alpha})^{1/2}\big)$.
\end{thm}

\begin{proof}
	Noticing the fact that $\cH_-\Phi_1 =  \cH_-\Phi_g = 0$.
	Actually, the operator $\mathcal H_{-}$ could be simplified below
	\beas
	\mathcal H_- =
	\begin{bmatrix} -\frac{1}{2} \nabla^2  + V(\bx) + \frac{\delta}{2} +  \betauu |\phiua|^2 + \, \betaud |\phida|^2 - \muu & 0 \\
		0 & -\frac{1}{2} \nabla^2  + V(\bx) - \frac{\delta}{2} + \, \betadu |\phiua|^2 + 2\betadd |\phida|^2 - \mud
	\end{bmatrix}.
	\eeas
	
	Then, according to Eqn. \eqref{NLEP2},
	we obtain that $\Phi_2$ lie in the nullspace of $\cH_-$, i.e., $\cH_-\Phi_2 = 0$.
	It is easy to check that
	$\lag\Phi_1,\Phi_2\rag = 0$ and $\|\Phi_2\| = 1$ with $c_{1}= -(\frac{1-\alpha}{\alpha})^{1/2},~c_{2} = (\frac{\alpha}{1-\alpha})^{1/2}$.
\end{proof}

\begin{remark}\label{rem:withoutJJ}
	From our extensive numerical results, not fully shown here,
	we conjecture that the following property holds
	\begin{center}
		$\cH_+$ is positive definite ~ ~ ~ and ~ ~ ~ $\ns(\cH_-) = \spa \{\Phi_1,\Phi_2\}$.
	\end{center}
\end{remark}

At last, we shall give the structure of the generalized nullspace of the operator $\mathcal{H}$.

\begin{thm}[\textbf{Generalized nullspace of $\cH$}]\label{thm:withoutJJ}
	If the above property is true,
	we can obtain the generalized nullspace of $\mathcal H$
	as follows
	\bea
	\ns(\cH^3) = \ns(\cH^2) = \ns(\cH) \oplus \spa \left\{\begin{bsmallmatrix}
		\widehat\Phi_1\\
		\mathbf 0
	\end{bsmallmatrix},
	\begin{bsmallmatrix}
		\widehat\Phi_2\\
		\mathbf 0 \\
	\end{bsmallmatrix}\right\},
	\eea
	where $\widehat{\Phi}_j = \cH_+^{-1}\Phi_j, ~j=1,2$ and  $\ns(\cH) = \spa \left\{\begin{bsmallmatrix}
			\mathbf 0 \\
			\Phi_1
		\end{bsmallmatrix},
		\begin{bsmallmatrix}
			\mathbf 0 \\
			\Phi_2
		\end{bsmallmatrix}\right\}$.
	
\end{thm}
%	The proof is similar to that of Theorem \ref{thm:withJJ}, and we omit it here.
\begin{proof}
Similar as Theorem \ref{thm:withJJ}, we can prove
$\ns(\cH) = \spa \left\{\begin{bsmallmatrix}
			\mathbf 0 \\
			\Phi_1
		\end{bsmallmatrix},
		\begin{bsmallmatrix}
			\mathbf 0 \\
			\Phi_2
		\end{bsmallmatrix}\right\}$,
\beas
\ns(\cH^2) &=& \{(\bbf,\bbg)^\top|\bm g\in \ns(\cH_-),\bbf\in\cH_+^{-1}\ns(\cH_-)\} \\
&=&\ns(\cH) \oplus \spa
\left\{\begin{bsmallmatrix}
		\widehat\Phi_1\\
		\mathbf 0
	\end{bsmallmatrix},
	\begin{bsmallmatrix}
		\widehat\Phi_2\\
		\mathbf 0
	\end{bsmallmatrix}\right\}.
\eeas
Then, for $\ns(\cH^3)$
\beas\ns(\cH^3) &=&\{(\bm f,\bm g)^\top|\cH_-\cH_+\cH_-\bm g=\mathbf 0, \cH_+\cH_-\cH_+\bm f=0\} \\
&=& \{(\bbf,\bbg)^\top|\cH_-\bm g\in \cH_+^{-1}\ns(\cH_-),\bbf\in\cH_+^{-1}\ns(\cH_-)\}.
\eeas
\vspace{-0.2em}
Obviously, $\ns(\cH^2)\subset \ns(\cH^3)$. Then we prove that $\ns(\cH^3)\subset \ns(\cH^2)$. For any $(\bbf,\bbg)^\top\in \ns(\cH^3)$, i.e., $\cH_-\bm g\in \cH_+^{-1}\ns(\cH_-)$, we have $\cH_-\bm g\in \cH_+^{-1}\ns(\cH_-)$, then $\cH_-\bm g = c_1\widehat\Phi_1+c_2\widehat\Phi_2$, and
\beas
	0&=&\langle \mathbf{0}, \bbg \rangle = \langle \mathcal H_-(c_1\Phi_1+c_2\Phi_2), \bbg \rangle\\
	&=&\langle c_1\Phi_1+c_2\Phi_2,\mathcal H_- \bbg \rangle\\
	&=&\langle c_1\Phi_1+c_2\Phi_2,c_1\widehat\Phi_1+c_2\widehat\Phi_2 \rangle \\
	&=& \langle \cH_+(c_1\widehat\Phi_1+c_2\widehat\Phi_2),c_1\widehat\Phi_1+c_2\widehat\Phi_2 \rangle.
\eeas
Combing the positivity of $\cH_+$, we have $c_1\widehat\Phi_1+c_2\widehat\Phi_2=\mathbf 0$, i.e., $\mathcal H_- \bbg=\mathbf 0$.
That is to say, $\ns(\cH^3) \subset \ns(\cH^2)$, and the proof is completed.
\end{proof}
The convergence of the non-zero eigenvalues depends heavily on the approximation accuracy of the generalized nullspace that is associated with $\cH$ \cite{LiWangZhang}.
Based on Theorems \ref{thm:withJJ} and \ref{thm:withoutJJ},  we obtain the generalized nullspace of $\cH$, thus essentially improving the convergence and efficiency.

%=================================================================================
\section{Numerical method}
\label{Numer_alg}

In this section, we will propose an efficient and spectrally accurate numerical method to solve the BdG equations \eqref{HpHm-eq}.
Due to the presence of external trapping potential $V(\bx)$, the ground states $\Phi_g(\bx)$
and eigenfunctions $(\bu,\bv)$ are all smooth and fast-decaying. Therefore, it is reasonable
to truncate the whole space $\mathbb{R}^d$
into a large enough bounded domain $\mathcal D\subset \mathbb R^d$ with periodic boundary conditions such that the truncation error is negligible.
Since all related functions are smooth,  the Fourier pseudo-spectral (FS) method stands out as the best candidate for spatial discretization \cite{Bao2017Mathematical,Bao2007mass,Tang2022Spectrally}
due to its simplicity, spectral accuracy and great efficiency.
The computation domain is usually chosen as a rectangle ${\mathcal D}_L:=[-L, L]^{d}$ and discretized uniformly with $N$ grid points in each spatial direction.

%Spatial discretization details will be presented in the next subsection.

\subsection{Spatial discretization by Fourier spectral method}
Provided that the stationary states $\Phi_g$ and all chemical potentials are precomputed with a very fine mesh by the PCG method \cite{ATZ-CiCP-PCG}, such that the numerical accuracy approaches machine precision.
For simplicity, we choose to illustrate the spatial discretization for the 1D case, and extensions to higher dimensions are omitted here.
We choose the computational domain ${\mathcal D}_L:=[-L,L]$ and discretize it uniformly with mesh size $h_x=\fl{2L}{N}$ where $N$ is a positive even integer. Define the grid point set as $\mathcal{T}_{N}=\{(-N/2,\cdots,N/2)h_x\}$
and introduce the plane wave basis as
\bes
W_{k}(x)=e^{\im \mu_k (x+L)},\quad  \mbox{for} ~ -N/2 \leq k \leq N/2-1,
\ees
with $\mu_k = \pi  k/L $.
Define $F_{n}:=F(x_n)$ ($F = \phi^g_j, u_j, v_j, V$ etc.) as function value at grid point $x_n\in\mathcal{T}_{N}$ and  $\mathbf F = (F_1, F_2,\ldots, F_N)$ as the corresponding discrete vector.
The FS approximations of $F$, denoted by ${\widetilde F}$, and its Laplacian $\nabla^2 F$ read as follows
\bea
F(x) \approx \widetilde{F}(x):=\sum_{k=-M/2}^{M/2-1}\widehat{\mathbf F}_{k}\; W_{k}(x),
\quad (\nabla^2 F)(x) \approx (\nabla^2 \widetilde F)(x)\! =\sum_{k=-N/2}^{N/2-1}-\mu_k^2~\widehat{\mathbf F}_{k}\; W_{k}(x),
\eea
where $\widehat{\mathbf F}_{k}$, the discrete Fourier transform of vector $\mathbf F \in \mathbb C^{N}$, is computed as
\begin{equation}
\label{eq:FT_3}
 \widehat{\mathbf F}_{k}= \frac{1}{N} \sum_{n=0}^{N-1} F_{n}\overline{W}_{k}(x_n)
 = \frac{1}{N} \sum_{n=0}^{N-1} F_{n} ~e^{\frac{-\im 2\pi n k}{N}}, \quad \quad    -N/2 \leq k \leq N/2-1,
  \end{equation} and is accelerated by a discrete Fast Fourier Transform (FFT) within  $O(N \ln(N))$ float operations. The numerical approximation of the Laplacian operator $\nabla^2$, denoted as $[\![\nabla^2]\!]$, corresponds to a dense matrix \cite{SpectralBkShen,Tang2022Spectrally,Zhang2021}.
 We map the function's pointwise multiplication $F_{n} G_{n}$ as a matrix-vector product $[\![F]\!] {\mathbf G}$, that is,
  \bea
\left( [\![F]\!] {\mathbf G} \right)_{n} := F_{n}G_{n} \quad \Longrightarrow \quad [\![F]\!] = \diag(\mathbf F).
 \eea
The operators $L_1,L_2$ in \eqref{def_B} is then discretized as follows
\beas
{[\![L_1]\!]} &=& -\frac{1}{2} [\![\nabla^2]\!]  + [\![V]\!] + \frac{\delta}{2} + 2\betauu [\![|\phiug|^2]\!] + \, \betaud [\![|\phidg|^2]\!] - \mu_g,\\
{[\![L_2]\!]} &=& -\frac{1}{2} [\![\nabla^2]\!]  + [\![V]\!] - \frac{\delta}{2} + \, \betadu [\![|\phiug|^2]\!] + 2\betadd [\![|\phidg|^2]\!] - \mu_g.
\eeas
Therefore, operators $\cA$ and $\cB$ are mapped into matrices $\mathbf A$ and $\mathbf B$ in the following way
\bea
\mathbf A := \begin{bmatrix}
{[\![L_1]\!]} & \betaud{[\![\phiug{\phidg}]\!]} + \frac{\Omega}{2} \\[0.4em]
\betadu{[\![{\phiug}{\phidg}]\!]} + \frac{\Omega}{2} & {[\![L_2]\!]}
\end{bmatrix},\quad \mathbf B := \begin{bmatrix}
\betauu[\![|\phiug|^2]\!] & \betaud[\![{\phiug}{\phidg}]\!]\\
\betadu[\![{\phiug}{\phidg}]\!] & \betadd[\![|\phidg|^2]\!]
\end{bmatrix}.
\eea

Thanks to the Fourier spectral discretization, both matrices $\mathbf A$ and $\mathbf B$ are symmetric, thus keeping their continuous operators' Hermitian property.
Setting $\bm{u}_N = (\bm{u}_1^N;\bm{u}_2^N)$ and $\bm{v}_N = (\bm{v}_1^N;\bm{v}_2^N)$,
the BdG equations \eqref{BdG_new} with constaint \eqref{constrain} are then discretized into a linear eigenvalue problem
\begin{gather}
\label{BdG-d}
\begin{bmatrix}
\mathbf{A}&\mathbf{B} \\
-\mathbf{B}&-\mathbf{A}
\end{bmatrix}
\begin{bmatrix}
\bm{u}_N \\
\bm{v}_N
\end{bmatrix}
=\omega_N\begin{bmatrix}
\bm{u}_N\\
\bm{v}_N
\end{bmatrix},\quad h_x\Big(\| \bm u_N \|_{l^2}^2 - \| \bm v_N \|_{l^2}^2\Big) = 1.
\end{gather}
With a similar change of variables, i.e., $\bm{f}_N = \frac{1}{2}(\bm{u}_N+\bm{v}_N),~\bm{g}_N = \frac{1}{2}(\bm{u}_N-\bm{v}_N)$, the above equation is rewritten into a linear response eigenvalue problem
\begin{gather}
\label{lrep-d}
\begin{bmatrix}
  \mathbf{H}_{+}&{\mathbf O} \\
{\mathbf O}&\mathbf{H}_{-}
\end{bmatrix}
\begin{bmatrix}
\bm{f}_N \\
\bm{g}_N
\end{bmatrix}
=\omega_N\begin{bmatrix}
  \mathbf O & \mathbf I \\
  \mathbf I & \mathbf O
\end{bmatrix}
\begin{bmatrix}
\bm{f}_N\\
\bm{g}_N
\end{bmatrix},
\end{gather}
where $\mathbf{H}_{+} = \mathbf{A}+\mathbf{B}$, $\mathbf{H}_{-} = \mathbf{A}-\mathbf{B}$. The above discrete dense eigensystem is solved using the recently developed BiOrthogonal Structure Preserving algorithm (BOSP for short) \cite{LiWangZhang}, and details are illustrated in the coming subsection.

%\begin{remark}[Matching eigenfunction constrain \eqref{constrainfg}]
%In practice, we {\sl do not}  have to bother about the constrain \eqref{constrainfg} when solving the eigenvalue problem \eqref{lrep-d}. Actually, once the eigenvectors are computed, the constraint \eqref{constrainfg} will be satisfied easily with a scalar scaling on $\bbf_N$ and $\bbg_N$.
%\end{remark}

\subsection{A bi-orthogonal structure-preserving method}\label{subsec:LREP-biorth}

As is known that ARPACK is a standard library to compute the discrete eigenvalue and eigenvector problem,
and it has been successfully adapted to BdG equations of dipolar BEC \cite{ARPACK,Tang2022Spectrally}.
However, a similar adaptation attempt using ARPACK is not as effective, and sometimes it does not even converge within a reasonable time,
because the generalized nullspace is much larger and more complicated, not to even mention the large-scale dense eigensystem for three-dimensional problems.
It is worth pointing out that the eigenspaces associated with eigenvalues of different magnitudes are biorthogonal to each other,
and such biorthogonality property shall be taken into account in the eigensolver design \cite{Bai2012Minimization,Bai2013Minimization,LiWangZhang} so as to achieve better performance.
The structure of generalized nullspace and biorthogonality, well described in Section \ref{sec:BdG_Prop} on the continuous level,
are assessed numerically by the linear algebra package (LAPACK)  and a modified Gram-Schmidt biorthogonal algorithm.

Based on the recently developed linear response solver BOSP \cite{LiWangZhang},
by combining the Fourier spectral method for spatial approximation and making use of the specific generalized nullspace structure (Section \ref{sec:rootspace}),
we propose an efficient iterative subspace eigensolver and provide a friendly interface for matrix-vector product evaluation.
Since there is no explicit matrix storage and the matrix-vector product is implemented via FFT within almost optimal
$O(N_{\rm t}\log N_{\rm t})$ operations ($N_{\rm t}:=N^d$ is the total number of grid points),
the heavy memory burden is much more alleviated and the computational efficiency is guaranteed to a large extent.
Therefore, it provides a feasible solution to large-scale problems, especially for such a densely populated eigensystem.
Thanks to the Fourier spectral method, our method can achieve spectral accuracy for both eigenvalues and eigenvectors
as long as the accuracy tolerance of BOSP is chosen sufficiently small.

As pointed out earlier, the matrix-vector product evaluation is the most time-consuming part. While in this article, the matrix-vector product
\begin{equation}
   \label{sol-dim}
   \mathbf{H}_{+}\bm f_N,\quad \mathbf{H}_{-}\bm g_N,\quad \mbox{for} \quad \bm f_N, \bm g_N \in  {\mathbb R}^{2N_{\rm t}},
\end{equation}
can be realized by two pairs of FFT/iFFT plus some function multiplication in the physical space.
The overall computational costs to compute the first $n_{\rm e}\in \mathbb Z^{+}$ eigenpairs amounts to
\beas
\mathcal{O}(n_{\rm e})\times \big(\mathcal{O}(N_{\rm t} \log(N_{\rm t})) + \mathcal{O}(N^d)\big) + \mathcal{O}(n_{\rm e}^3)
=  \mathcal{O}(n_{\rm e}~N_{\rm t}\log(N_{\rm t}))  + \mathcal{O}(n_{\rm e}^3).
\eeas
When the degree of freedom is much larger than the number of eigenvalues, i.e., $ N_{\rm t} \gg n_{\rm e}$,
the computational costs are approximately $\mathcal{O}(N_{\rm t}\log(N_{\rm t}))$ and will be verified numerically in the next section.

\subsection{Stability analysis and error estimates}
To demonstrate the practicality of our solver,  we will present a rigorous stability analysis for the eigenfunctions in this subsection.
Since the density and eigenfunction both decay fast enough in space,
it is reasonable to assume that ${q}(\bx)$ (${q} = \phi^g_j, u_j, v_j$, etc.) is numerically compactly supported
in a bounded domain $\Omega\subsetneq\mathbb{R}^d$, that is,  $\mathrm{supp}\{{q}\} \subsetneq \Omega$. We introduce $2L$-periodic Sobolev space $H^m_p(\Omega)
\subset H^m(\Omega) \ \ (m\geq1)$ with $\Omega = \mathcal{D}_L$, and the semi-norm and norm as follows, respectively,
$$|{q}|_{m} := \left(\sum\nolimits_{|\bm{\alpha}|=m} \|\partial^{\bm{\alpha}} {q} \|^2\right)^{1/2},
\quad  \|{q}\|_{m} := \left(\sum\nolimits_{ |\bm{\alpha}|\leq m} \|\partial^{\bm{\alpha}} {q} \|^2\right)^{1/2},$$
with index $\bm\alpha=(\alpha_1,\ldots, \alpha_d)\in\mathbb{Z}^d$, $|\bm\alpha|=\sum_{j=1}^d\alpha_j$, $\partial^{\bm\alpha}=\partial_{x_1}^{\alpha_1}\cdots\partial_{x_d}^{\alpha_d}$ and $\|\cdot\|$ being the $L^2$ norm.
For vector-valued function $\bm {q} = ({q}_{1},{q}_{2})^\top$, we define $|\bm {q}|_{m} := \sqrt{\sum_{j}|{q}_j|_{m}^2}$,~ $\|\bm {q}\|_{m} := \sqrt{\sum_{j}\|{q}_j\|_{m}^2}$,~and $\|\bm {q}\| :=\sqrt{\sum_{j}\|{q}
_j\|^2}$.
We use $A\lesssim B$ to denote $A \leq cB$ where the
constant $c>0$ is independent of $N$. Then we first introduce the stability analysis as follows:

\begin{thm}[\textbf{Stability Analysis}]\label{thm:stability}
 For any $\bm{{q}}(\bx)=({q}_1(\bx),{q}_2(\bx))^\top \in (H_p^{m}(\Omega))^2$  with $m> d/2 + 2$, and its Fourier pseudo-spectral approximations $\widetilde{\bm{{q}}}(\bx)=\big(\widetilde{q}_1(\bx),\widetilde{q}_2(\bx)\big)^\top$, we have the following error estimate
\bea
\|\mathcal{Q} \bm{{q}}- \mathcal{Q}\widetilde{\bm{{q}}}\| & \lesssim & N^{-(m - 2)},
\eea
where $\mathcal{Q} = \mathcal{A},\mathcal{B}$, and $\bm{{q}}=\bm{u},\bm{v}$.
\end{thm}
To prove the above theorem, the following preparations are required.
\begin{lem}[\textbf{Fourier spectral approximation} \cite{LiuZhang_opt,SpectralBkShen}] \label{FouApprox}
 For ${q}(\bx) \in H_p^{m}(\Omega)$  with $m> d/2$, and its Fourier pseudo-spectral approximation $\cI {q}$,
 we have the following error estimates
\bea\label{approx_lem_2}
\| \partial^{\bm{\alpha}}( {q}- \widetilde{q} )\| & \lesssim & N^{-(m - |\bm{\alpha}|)} |{q}|_{m} ,\quad\quad ~0\leq |\bm\alpha|\leq m.
\eea
\end{lem}

Since operators $\mathcal A$ and $\mathcal B$ contain only two types of operators, namely, the Laplacian operator $\nabla^2$
and function multiplication operation that involves  $V(\bx)$ \& $\phi_i^g\phi_j^g$, the following corollary is  a prerequisite for proving Theorem \ref{thm:stability}.
\begin{corollary}\label{coro_approx}
 For ${q}(\bx) \in H_p^{m}(\Omega)$  with $m> d/2$, and its Fourier pseudo-spectral approximation $\cI {q}$, we have
\bea
\|\nabla^2 {q}- \nabla^2 \widetilde{q} \| & \lesssim & N^{-(m - 2)} |{q}|_{m},\\ \label{error_1}
\|V {q}- V \widetilde{q} \|  & \lesssim & N^{-m } |{q}|_{m},\\ \label{error_2}
\|{\phi_j^g}{\phi_l^g} {q}- {\phi_j^g}{\phi_l^g} \widetilde{q} \|  & \lesssim & N^{-m } |{q}|_{m} \ \ (j,l=1,2).\label{error_3}
\eea
\end{corollary}
It is easy to prove the above corollary using Lemma \ref{FouApprox}, and we shall omit details for brevity.

\vspace{0.2cm}

\noindent {\bf The proof of Theorem \ref{thm:stability}}:
\vspace{-0.2cm}
\begin{proof}
For simplicity, we only prove the $\mathcal Q=\mathcal A$ case and extensions to other operators are similar.
From \eqref{def_B} and \eqref{approx_lem_2}-\eqref{error_2}, we obtain
$$\|(-1/2\nabla^2)\mathbf{I}_2(\bm {q})-(-1/2\nabla^2)\mathbf{I}_2(\widetilde{\bm {q}})\| \lesssim N^{-(m-2)}|\bm {q}|_{m},$$
where $\mathbf I_2$ is the second-order unit diagonal matrix.

For the other term of $\mathcal A$ as shown in Eqn. \eqref{def_B}, using \eqref{approx_lem_2}, \eqref{error_3} and triangle inequality,
the following error estimate holds:

\beas
\left\|\big(\mathcal{A}-(-1/2\nabla^2)\mathbf{I}_2\right)\bm {q} - \left(\mathcal{A}-(-1/2\nabla^2)\mathbf{I}_2\big)\widetilde{\bm {q}}\right\|\lesssim & N^{-m}|\bm {q}|_{m}.
\eeas
Therefore, we finish the proof of Theorem \ref{thm:stability}.
\end{proof}

\

Then we will propose the error estimates for the FS approximations. To begin with, and adopt the Sobolev space $\mathbb{V} = (H^1_p(\Omega))^2\times (H^1_p(\Omega))^2$ with the following norm
$$\|\bm \Phi\|_{1} = \sqrt{\|\bm f\|_{1}^2+\|\bm g\|_{1}^2}~\mbox{ and }~ \|\bm \Phi\| = \sqrt{\|\bm f\|^2+\|\bm g\|^2}, \quad ~\forall \,\bm\Phi:= (\bm f;\bm g) \in \mathbb V.$$

%\vspace{0.5em}
The BdG equation \eqref{HpHm-eq} is equivalent to the following eigenvalue problem:
\bea \label{HpHm-eq-1}
\widetilde\cH\begin{bmatrix}
	\bbf  \\ \bbg
	\end{bmatrix} :=
\begin{bmatrix}
\mathcal H_+&\mathcal O\\
	\mathcal O&\mathcal H_-
\end{bmatrix}
\begin{bmatrix}
	\bbf  \\ \bbg
	\end{bmatrix} =
	\omega \begin{bmatrix}
0&1\\
	1&0
\end{bmatrix}
\begin{bmatrix}
	 \bbf  \\ \bbg
\end{bmatrix}=:
	\omega \mathcal J
\begin{bmatrix}
	 \bbf  \\ \bbg
\end{bmatrix}.
\eea
The Galerkin weak form of Eqn.~\eqref{HpHm-eq-1} reads as: to find $0\neq \omega\in \mathbb{R}$ and $\bm \Phi=(\bm f;\bm g) \in \mathbb{V}$ such that
\bea\label{BdG-wf}
a(\bm \Phi,\bm\Psi) = \omega\, b(\bm \Phi,\bm\Psi), \ \ \ \forall\, \bm\Psi= (\bm\xi;\bm\eta)\in \mathbb{V},
\eea
subject to constraint $b(\bm \Phi,\bm \Phi) = 1/2$ with
\beas
a(\bm \Phi,\bm\Psi) := \lag \cH_+ \bm f,\bm\xi\rag  + \lag \cH_-\bm g,\bm\eta\rag, \quad b(\bm \Phi,\bm\Psi) := \lag\bm g,\bm\xi\rag + \lag\bm f,\bm\eta\rag.
\eeas

\begin{remark}
From Remark \ref{rem:withJJ} (or \ref{rem:withoutJJ}), we can prove that there exists $c_0>0$ such that
\begin{equation}\label{coer_BdG}
a(\bm \Phi,\bm\Phi) \geq c_0\|\bm \Phi\|_{1},\quad \forall\, \bm\Phi \in \mathbb{U},
\end{equation}
where the subspace $\mathbb{U}$ is defined as
$$ \mathbb{U} = \ns(\cH_+)^\bot\times \ns(\cH_-)^\bot \subset \mathbb{V},$$
with $\ns(\cH_\pm)^\bot := \{\bm f \in (H_p^1(\Omega))^2\, |\, \lag\bm f,\bm\xi\rag = 0, ~\forall\, \bm\xi \in \ns(\cH_\pm)\}$.
It is important to point out that solving the BdG equation \eqref{BdG-wf} in $\mathbb{U}$ is equivalent to solving non-zero $\omega$ and its associated eigenfunction $\bm{\Phi}$ in $\mathbb{V}$.
\end{remark}
Define the approximation finite dimensional spaces $X_N$ of $H^1_p(\Omega)$ and $\mathbb{V}_N$ of $\mathbb{V}$ as follows:
$${X}_N:= \spa\left\{W_k(\bx), k = -N/2,\ldots, N/2-1\right\},\quad \mathbb{V}_N := ({X}_N)^2\times ({X}_N)^2,$$
and consider its approximation problem: to find $0\neq \omega_N \in \mathbb{R}$ and $\bm \Phi_N=(\bm f_N;\bm g_N)\in \mathbb{V}_N$ such that
\bea\label{BdG-wf-a}
a_N(\bm \Phi_N,\bm\Psi_N) = \omega_N\,b_N(\bm \Phi_N,\bm\Psi_N), \ \ \ \forall\, \bm\Psi_N=(\bm \xi_N;\bm \eta_N)\in \mathbb{V}_N,
\eea
subject to normalization constraint $b_N(\bm \Phi_N,\bm\Phi_N)={1}/{2}$ with
\beas
a_N(\bm \Phi_N,\bm\Psi_N) = \lag\cH_+\bm f_N,{\bm\xi_N}\rag_N+ \lag\cH_-\bm g_N,{\bm\eta_N}\rag_N, \quad
b_N(\bm \Phi_N,\bm\Psi_N) = \lag\bm g_N,{\bm\xi_N}\rag_N + \lag\bm f_N,{\bm\eta_N}\rag_N,
\eeas
 and  $\lag\bm f,\bm \xi\rag_N:= \sum_{j}\lag f^N_j,\xi^N_j\rag_N := \sum_{j} \big(\frac{2L}{N}\sum_{n=0}^N{f_j^N}(\bx_n)\overline{g_j^N(\bx_n)}\big)$.

 \

\noindent Similar to \cite[Lemma A.1]{LLXZ-BdGSpin1}, we have the following equivalence result and choose to omit proofs for brevity.
\begin{lem}
The discrete problem \eqref{lrep-d} and the discrete variational problem \eqref{BdG-wf-a} are equivalent.
\end{lem}

To illustrate the convergence behavior, we introduce the following notation
\begin{equation}\label{infSupCondVecFun}
\delta_N(\bm\Phi) = \inf_{\bm\Psi\in \mathbb{V}_N}\bigg\{\|\bm\Phi-\bm\Psi\|_1 + \sup_{\bm \Theta\in \mathbb{V}_N}\frac{|a(\bm\Psi,\bm \Theta)-a_N(\bm\Psi,\bm \Theta)|}{\|\bm \Theta\|_1} \bigg\}.
\end{equation}
Similar as \cite[Lemma A. 2]{LLXZ-BdGSpin1} for scalar function,
by the coerciveness of bilinear operator $a(\cdot,\cdot)$ (Eqn. \eqref{coer_BdG}) and the conformal subspace approximation, we obtain
\begin{equation}
\delta_N(\bm\Phi)\lesssim N^{-(m-\sigma)},\qquad \mbox{with}  \quad \sigma = \max\{1,d/2\}.
\end{equation}
According to the general theory given by \cite[Lemma A. 3]{LLXZ-BdGSpin1}, we derive the following error estimates.

\begin{lem}\label{lem:ErrEsti1}
For any eigenpair approximation $\{\omega_N;\bm\Phi_N\}$ of Eqn. \eqref{BdG-wf-a}, there is
an eigenpair $\{\omega;\bm \Phi\}$ of Eqn. \eqref{BdG-wf} corresponding to $\omega$ such that
\beas
\|\bm \Phi-\bm\Phi_N\|_{1}\lesssim \delta_N(\bm\Phi),\quad
\|\bm \Phi-\bm\Phi_N\| \lesssim \zeta_N\|\bm \Phi-\bm\Phi_N\|_{1},
\quad |\omega-\omega_N| \lesssim \|\bm \Phi-\bm\Phi_N\|_{1}^2,\label{convee}
\eeas
where the constant $\zeta_N$ approaches $0$ as $N\rightarrow \infty$.
\end{lem}

Finally, via a change of variables, we obtain error estimates for $(\bm u,\bm v)$.
\begin{thm}[\textbf{Error Estimates}]\label{thm:error-estimate}
If $\bm u,\bm v \in (H_p^{m}(\Omega))^2$  with $m\geq1$ and $\mathrm{supp}\{\bu\}, \mathrm{supp}\{\bv\} \subsetneq \Omega$, then for any eigenpair approximation $\{\omega_N;\bm{u}_N,\bm{v}_N\}$ of \eqref{BdG-d}, there is
an eigenpair $\{\omega;\bm u,\bm v\}$ of Eqn.~\eqref{BdG_new} satisfying the
following error estimates
\beas
\|\bm{u}-\bm{u}_N\|+\|\bm{v}-\bm{v}_N\| &\lesssim& N^{-(m-\sigma)}, \\
|\omega-\omega_N| &\lesssim& N^{-2(m-\sigma)},
\eeas
where $\sigma = \max\{1,d/2\}$.
\end{thm}

\section{Numerical results}
\label{NumResult}
In this section, we shall first carry out a comprehensive numerical investigation to illustrate the accuracy and efficiency.
Then we apply our method to study the Bogoliubov excitations around the ground states.
The ground states $\Phi_g$ and its associated chemical potentials are computed with a very fine mesh size on a large enough
domain using the PCG method, thus achieving an up to almost machine-precision accuracy.
Unless stated otherwise, we choose harmonic trapping potential $V(\bx)$ as described in \eqref{Vpoten} and
denote the trapping frequencies by a vector $\vec{\bm \gamma}  \in \mathbb R^d$, for example, $\vec{\bm\gamma} = (1,2,3)$ represents
$\gamma_x = 1, \gamma_y = 2,\gamma_z=3$ for 3D harmonic trap.
The computation domain is chosen as a rectangular $\mathcal{D}_L=[-L,L]^d$
and discretized uniformly with the same $N \in 2\mathbb Z^{+}$ grid points in each spatial direction.
The mesh sizes are chosen the same in each direction and denoted as $h$ for simplicity.
Throughout this section, we shall take $L = 16$ for the 1D and 2D problems and $ L = 8$ in the 3D computation.

\

The relative error of the eigenvalue is expected to satisfy a prescribed tolerance of $10^{-9}$ throughout this section.
We shall compute BdG excitations of GPEs \eqref{GPEs} with Josephson junctions (\textbf{JJ}) (i.e., $\Omega = 1$ and $\delta = 0$)
and without Josephson junctions (i.e., $\Omega = 0,\delta=0$ and $\alpha=0.2$)
where the interaction strength coefficients are chosen as $\betauu=100, \betaud=94, \betadd=97$.
The algorithms were implemented in Fortran (ifort version 2021.3.0) and run on a 3.00 GHz Intel(R) Xeon(R) Gold 6248R CPU with a 36 MB cache in Ubuntu GNU/Linux.

\subsection{Performance investigation}
\label{Accuracy tests}
In this subsection, we study the performance in terms of accuracy and efficiency in different space dimensions.
Analytical eigenvalues and eigenvectors were given for the harmonic trapping potential by Eqn. \eqref{eq:eig-pair} in {Theorem \ref{lemAnalytical}}.
For an eigenvalue $\og$ with multiplicity $K$, we denote its associated analytical eigenspaces as $\mathcal M_{\bm{u}}:={\rm span} \{\bm{u}_1, \cdots, \bm{u}_K\}, \mathcal M_{\bm{v}}:={\rm span} \{\bm{v}_1, \cdots, \bm{v}_K\}$.
The numerical approximations of $\{\omega_{\sigma};\bm{u}_{\sigma}, \bm{v}_{\sigma}\}$ are denoted by
$\{\omega_{\sigma}^{N};\bm{u}_{\sigma}^{N}, \bm{v}_{\sigma}^{N}\}$ with ${\sigma}\in \{x,y,z\}$.
To demonstrate the results, we adopt the following error functions
\bea
\err_{\og}^{\sigma}:= \fl{|\omega_{\sigma}^N-\omega_{\sigma}|}{|\omega_{\sigma}|},
\qquad
\err_{\bm{uv}}^{\sigma}:=\fl{\| \bm{u}_{{\sigma}}^{N}-\mathcal{P}^N_{\bm{u}}  \bm{u}_{{\sigma}}^{N}\|_{l^2}}{\| \bm{u}_{{\sigma}}^{N} \|_{l^2}}
+ \fl{\| \bm{v}_{{\sigma}}^{N}-\mathcal{P}^N_{\bm{v}} \bm{v}_{{\sigma}}^{N}\|_{l^2}}{\|\bm{v}_{{\sigma}}^{N}\|_{l^2}},
\eea
where $\|\cdot\|_{l^2}$ is the discrete $l^2$ norm and $\mathcal{P}_\nu^N$ ($\nu=\bm{u},\bm{v}$)
is the $l^2$-orthogonal projection operator into space $\mathcal M_\nu^N$
which is the discrete version of $\mathcal{M}_\nu$ at mesh grid $\mathcal T_N$.

\begin{exam}[\textbf{Accuracy}]
\label{Exam_Accuracy_Test}
We study the accuracy behavior for BdG equation \eqref{BdG_uv}-\eqref{BdG_uv-woJJ} with/without Josephson junction in 1D/2D/3D.
To this end, we consider the following four cases:
\begin{itemize}
\item[] \hspace{-0.5cm}{\bf Case I.}  ~1D case: $\gamma_x = 1$;
\item[] \hspace{-0.5cm}{\bf Case II.}  Isotropic trap in 2D:  $\gamma_x = \gamma_y=1 $;
\item[]\hspace{-0.65cm} {\bf Case III.} Anisotropic trap in 2D: $\gamma_x =\gamma_y/2=1$;
\item[] \hspace{-0.5cm}{\bf Case IV.} Isotropic trap in 3D:  $\gamma_x=\gamma_y=\gamma_z=1 $.
\end{itemize}

\end{exam}

\noindent

As proved in Lemma \ref{lemAnalytical}, there exist $d$-pairs of analytical eigenvalue/vectors in $d$-dimensional problems for all cases
and the multiplicity of eigenvalues for isotropic trap (i.e., \textbf{Case II and IV}) is $d$.
Table \ref{Err_1D}-\ref{Err_3D} present the numerical errors of eigenvalues and eigenvectors
that are obtained with different mesh sizes in \textbf{Case I-IV}.
We can observe that:
(1) the numerical solutions manifest a typical spectral convergence for both eigenvalue and eigenvector;
(2) the approximation errors saturate as the mesh size tends smaller
and the saturated accuracy is well below the preset tolerance ($10^{-9}$) of our eigenvalue solver,
including both the 2D and 3D problem (i.e., \textbf{Case II-IV}).

\begin{table}[!htb]
\def\temptablewidth{1\textwidth}
\tabcolsep 0pt	\caption{Errors of the eigenvalue/eigenvector for {\bf Case I} in \textbf{1D}
for with/without \textbf{JJ} in Example \ref{Exam_Accuracy_Test}.}
	\label{Err_1D}\small
	\begin{center}\vspace{-1.5em}
{\rule{\temptablewidth}{1pt}}
\begin{tabularx}{\temptablewidth}{@{\extracolsep{\fill}}ccccccccc}
   & $N$ &$32$& $64$ & $128$ & $256$ & $512$ \rule{0pt}{10pt}\\[0.2em]
\hline
\rule{0pt}{12pt}
\multirow{2}{*}{ $\Omega = 1$}   &  $ \err_{\og}^x $                &1.917E-02 & 4.786E-06 & 1.490E-10 & 4.441E-15 & 7.994E-15\\[0.1em]
	 & $ \err_{\bm{uv}}^{x} $   										 & 3.660E-02 & 1.077E-03 & 4.103E-07 & 9.817E-12 &  1.948E-11
 \\ [0.2em]
\hline
	\rule{0pt}{12pt}
\multirow{2}{*}{ $\Omega =0$}  & $ \err_{\og}^x $              &8.391E-02 & 2.398E-06 &  1.461E-10 & 4.249E-12 & 1.827E-13\\[0.1em]
   &  $ \err_{\bm{uv}}^{x} $   										 &1.728~~~~~~~& 1.077E-03 &  1.461E-07 & 1.675E-10 & 1.997E-11\\
\end{tabularx}
{\rule{\temptablewidth}{1pt}}
\end{center}
\end{table}

\begin{table}[!htb]
\def\temptablewidth{1\textwidth}
\tabcolsep 0pt	\caption{Errors of the eigenvalue/eigenvector for {\bf Case II} (upper) and {\bf III} (lower) in \textbf{2D}
with/without \textbf{JJ} in Example \ref{Exam_Accuracy_Test}.}
\label{Err_2D}\small
\begin{center}
\vspace{-1.5em}
{\rule{\temptablewidth}{1pt}}
\begin{tabularx}{\temptablewidth}{@{\extracolsep{\fill}}ccccccc}
 \multicolumn{7}{c}{\textbf{Case II:} isotropic trap ($\gamma_x = \gamma_y = 1$)} \rule{0pt}{12pt}\\ [0.2em]\hline
   & \makecell[cc]{$N$} &$32$& $64$ & $128$ & $256$ & $512$ \rule{0pt}{10pt}\\[0.2em]
\hline
\rule{0pt}{12pt}
&$ \err_{\og}^{x}$  				              &1.007E-02 & 1.742E-06 & 9.452E-13 & 1.028E-13 & 3.786E-13\\[0.1em]
\multirow{2}{*}{$\Omega=1$}& $\err_{\og}^{y}$         &1.007E-02 & 1.742E-06 & 9.484E-13 & 1.223E-13 & 3.901E-13\\[0.4em]
& $ \err_{\bm{uv}}^{x}$             			      &6.403E-02 & 7.560E-04 & 4.747E-08 & 9.157E-11 & 3.626E-11\\[0.1em]
&$ \err_{\bm{uv}}^{y} $     		                      &5.525E-02 & 6.239E-04 & 3.825E-08 & 4.187E-11 & 4.246E-11\\[0.2em]
\hline
\rule{0pt}{12pt}
                &$ \err_{\og}^{x}$            		     &1.007E-02 & 1.765E-06 & 9.137E-13 & 2.620E-14 & 1.914E-13  \\[0.1em]
\multirow{2}{*}{ $\Omega =0$}  &$ \err_{\og}^{y} $     &1.007E-02 & 1.765E-04 & 9.526E-13 & 3.086E-14 & 2.041E-13  \\[0.4em]
 &$ \err_{\bm{uv}}^{x} $        		             &5.910E-02 & 6.697E-04 & 6.493E-08 & 5.445E-12 & 1.590E-11  \\[0.1em]
 & $ \err_{\bm{uv}}^{y} $       		     &5.131E-02 & 5.470E-04 & 5.261E-08 & 2.671E-12 & 3.810E-11  \\[0.2em]
 \midrule

 \multicolumn{7}{c}{\textbf{Case III: } anisotropic trap ($\gamma_x = 1,\gamma_y =2$) } \rule{0pt}{10pt}\\ [0.3em] \hline
\rule{0pt}{12pt}
 & $N$ &$32$& $64$ & $128$ & $256$ & $512$\\[0.2em]
\hline
\rule{0pt}{12pt}
                &$\err_{\og}^x $                	&7.184E-03 & 8.939E-06 & 1.954E-13 & 2.887E-14 & 3.124E-13 \\[0.1em]
\multirow{2}{*}{ $\Omega =1$}      &$\err_{\og}^y $ 	&1.152E-02 & 5.077E-04 & 1.133E-08 & 2.442E-15 & 1.981E-13 \\[0.4em]
	     &$ \err_{\bm{uv}}^{x}$         		&4.331E-02 & 7.871E-04 & 1.191E-07 & 8.814E-12 & 1.913E-12 \\[0.1em]
                &$ \err_{\bm{uv}}^{y}$ 			&2.659E-02 & 9.205E-03 & 1.118E-05 & 1.391E-11 & 2.451E-12 \\ [0.2em]
\hline
	\rule{0pt}{12pt}
                &$ \err_{\og}^x $              		&7.177E-03&  8.964E-06 & 1.936E-13 & 7.771E-14 & 1.352E-13 \\[0.1em]
\multirow{2}{*}{ $\Omega = 0$ }     &$ \err_{\og}^y $ &1.149E-02&  5.084E-04 & 1.132E-08 & 5.240E-13 & 8.970E-14 \\ [0.4em]
&$ \err_{\bm{uv}}^{x} $         			&4.421E-02&  7.876E-04 & 1.190E-07 & 7.059E-12 & 1.954E-11 \\[0.1em]
                &$ \err_{\bm{uv}}^{y} $         	&2.658E-01&  9.206E-03 & 1.118E-05 & 1.915E-11 & 5.950E-12 \\[0.2em]

\end{tabularx}
{\rule{\temptablewidth}{1pt}}
\end{center}
\end{table}

\begin{table}[!htb]
\def\temptablewidth{1\textwidth}
\tabcolsep 0pt	\caption{Errors of the eigenvalue/eigenvector for {\bf Case IV} in \textbf{3D} with/without \textbf{JJ}
in Example \ref{Exam_Accuracy_Test}.}
\label{Err_3D}\small
\begin{center}
\vspace{-1.5em}
{\rule{\temptablewidth}{1pt}}
\begin{tabularx}{\temptablewidth}{@{\extracolsep{\fill}}cccccccc}
                & $N$                               &$16$   & $32$   & $64$   & $128$  \rule{0pt}{10pt}\\[0.2em]
\hline
	\rule{0pt}{12pt}
                &$ \err_{\og}^x $         		   		  &5.883E-03 & 2.880E-06 & 2.089E-13 & 1.060E-12 \\  [0.1em]
                &$ \err_{\og}^y $  			     		  &5.883E-03 & 2.880E-06 & 2.103E-13 & 1.062E-12 \\  [0.1em]
\multirow{2}{*}{ $\Omega=1$ }       &$ \err_{\og}^z $  	  &5.883E-03 & 2.880E-06 & 2.107E-13 & 1.064E-12 \\  [0.4em]
  &$ \err_{\bm{uv}}^{x} $        						  &7.270E-02 & 5.609E-04 & 2.575E-08 & 8.528E-11 \\  [0.1em]
                &$ \err_{\bm{uv}}^{y} $      		      &7.275E-02 & 5.608E-04 & 2.575E-08 & 4.695E-11 \\  [0.1em]
                &$ \err_{\bm{uv}}^{z} $     		  	      &7.271E-02 & 5.608E-04 & 2.575E-08 & 4.120E-11 \\  [0.2em]
\hline
	\rule{0pt}{12pt}
                &$\err_{\og}^x $         			      &5.895E-03 & 2.881E-06 & 5.730E-13 & 1.798E-12 \\  [0.1em]
                &$\err_{\og}^y $         		  		  &5.895E-03 & 2.881E-06 & 5.798E-13 & 1.811E-12 \\  [0.1em]
\multirow{2}{*}{ $\Omega=0$}  &$\err_{\og}^z $     &5.895E-03 & 2.881E-06 & 5.802E-13 & 1.820E-12 \\  [0.4em]
  &$ \err_{\bm{uv}}^{x} $         						  &7.273E-02 & 5.610E-04 & 2.572E-08 & 1.086E-11 \\  [0.1em]
                &$ \err_{\bm{uv}}^{y} $      			  &7.269E-02 & 5.610E-04 & 2.572E-08 & 1.372E-11 \\  [0.1em]
                &$ \err_{\bm{uv}}^{z} $        			  &7.273E-02 & 5.610E-04 & 2.572E-08 & 1.680E-11 \\  [0.2em]
\end{tabularx}
{\rule{\temptablewidth}{1pt}}
\end{center}
\end{table}

\begin{exam}[\bf Efficiency]\label{Exam_Efficiency_Test}
We investigate the efficiency performance by computing the first $20$ positive eigenvalues
and the corresponding eigenfunctions with/without Josephson junction in 3D.
To this end, we consider the following cases
\begin{itemize}
   \item[] \hspace{-0.5cm}{\bf Case I.}  With \textbf{JJ} for isotropic trap with $\vec{\bm{\gamma}}=(1,1,1)$
      and anisotropic trap with $\vec{\bm{\gamma}} = (1,1,2)$;
   \item[] \hspace{-0.5cm}{\bf Case II.} Without \textbf{JJ} for isotropic trap with $\vec{\bm{\gamma}}=(1,1,1)$
      and anisotropic trap with $\vec{\bm{\gamma}} = (1,1,2)$.
\end{itemize}

\end{exam}

\begin{figure}[!htp]
\centering
\includegraphics[width=8.15cm,height=6.5cm]{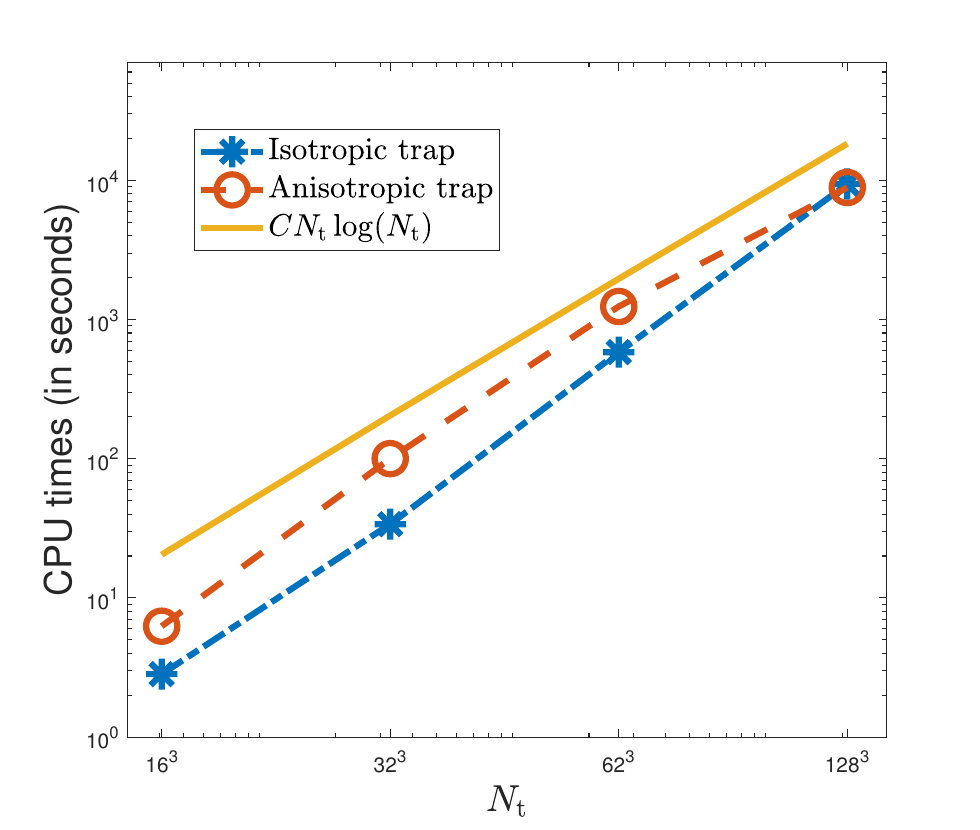}
\includegraphics[width=8.15cm,height=6.5cm]{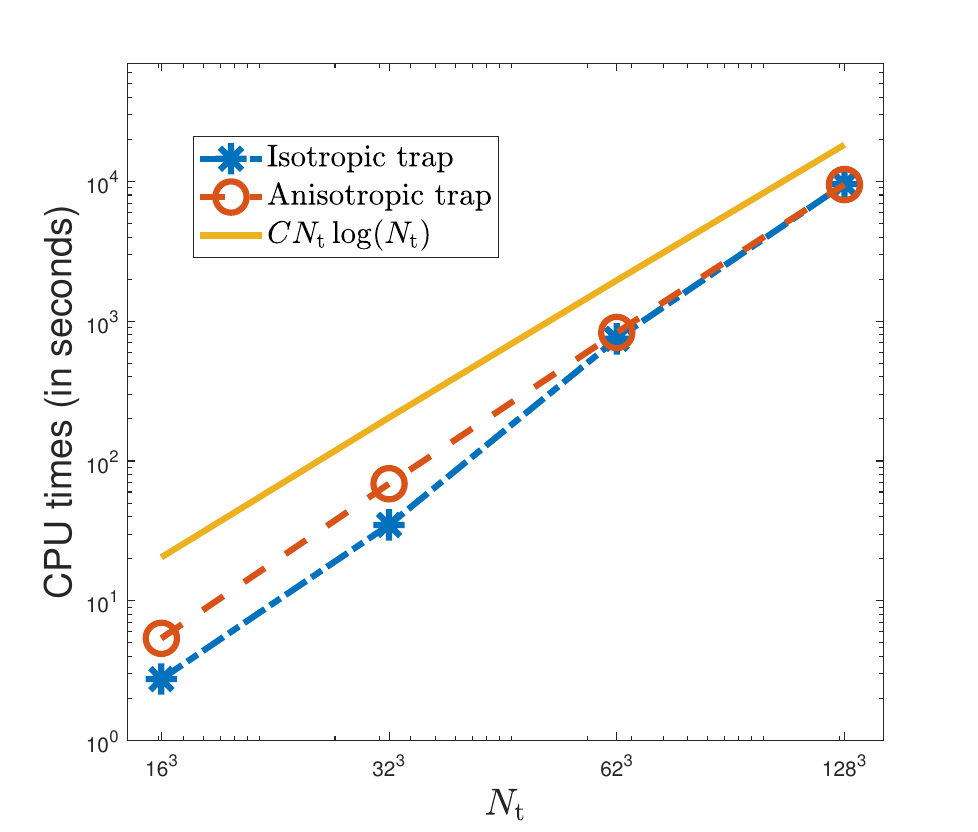}
\caption{The computational time versus degrees of freedom for {\bf Case I} (left) and {\bf Case II} (right) in Example \ref{Exam_Efficiency_Test}.}
\label{times_efficiency_test}
\end{figure}

Figure \ref{times_efficiency_test} presents the computation time (measured in seconds) versus degrees of freedom for 3D problems
as described above in Example \ref{Exam_Efficiency_Test}.
The degree of freedom varies as $16\,384,~131\,072,~1\,048\,576,~8\,388\,608$,
while the number of eigenvalue is kept unchanged as $n_{\rm ev} = 20$ with a tolerance $10^{-9}$.
From Figure \ref{times_efficiency_test}, we can see clearly that the degree of freedom is much larger than the number of desired eigenvalues,
and the overall computation time scales as $\mathcal{O}(N_{\rm t}\log N_{\rm t})$,
which agrees well with the complexity of matrix-vector multiplication as expected.
Given the fact that the discrete BdG system is non-Hermitian and dense, and the first $20$ eigenpairs are computed with 8 million
degrees of freedom within 3 hours using sequential CPU computing, it manifests a very promising potential in physical applications.

%%%%%-------------------------------%%%%%

\subsection{Applications of 2D cases}

In this subsection, we to investigate the excitation spectrum
and Bogoliubov amplitudes of two-component BEC around the ground state.
To visualize a normal mode, similar to \eqref{waveAssump}, we analyze the evolution of the perturbed density profile \cite{YiLowLying18}
\bea
{n}^\ell_j(\bx,t)=\left|\left[\phi_j(\bx)+ \varepsilon\big(u^\ell_j(\bx)
e^{-i\og_\ell t}+\bar{v}_j^\ell(\bx)e^{i\og_\ell t}\big) \right]\right|^2,\ \ j= 1,2,%\ \ \ell = 1,2,\ldots,
\eea
which reveals the nature of excitations indexed $\ell$. To this end, we use a very fine mesh size $h = 1/8$.

\begin{exam}
\label{perturb_2D_1}
We investigate the perturbed density evolution by different excitation modes with/without a Josephson junction in 2D.
To this end, we set perturbation strength $\varepsilon=0.1$ and consider the following four cases:
\begin{itemize}
\item[] \hspace{-0.5cm}{\bf Case I.} ~~Isotropic trap with \textbf{JJ}: $\Omega =1$, $\gamma_x = \gamma_y=1$, $\ell=15$;
\item[] \hspace{-0.5cm}{\bf Case II.} Anisotropic trap with \textbf{JJ}: $\Omega =1$, $\gamma_x = \gamma_y/2 =1$, $\ell=12$;
\item[] \hspace{-0.5cm}{\bf Case III.} Isotropic trap without  \textbf{JJ}: $\Omega=0$, $\gamma_x = \gamma_y=1$, $\ell=14$;
\item[] \hspace{-0.5cm}{\bf Case IV.} Anisotropic trap without \textbf{JJ}: $\Omega=0$, $\gamma_x = \gamma_y/2 =1$, $\ell=16$.
\end{itemize}
\end{exam}

\begin{figure}[!htp]
\centering
  \includegraphics[scale=0.38]{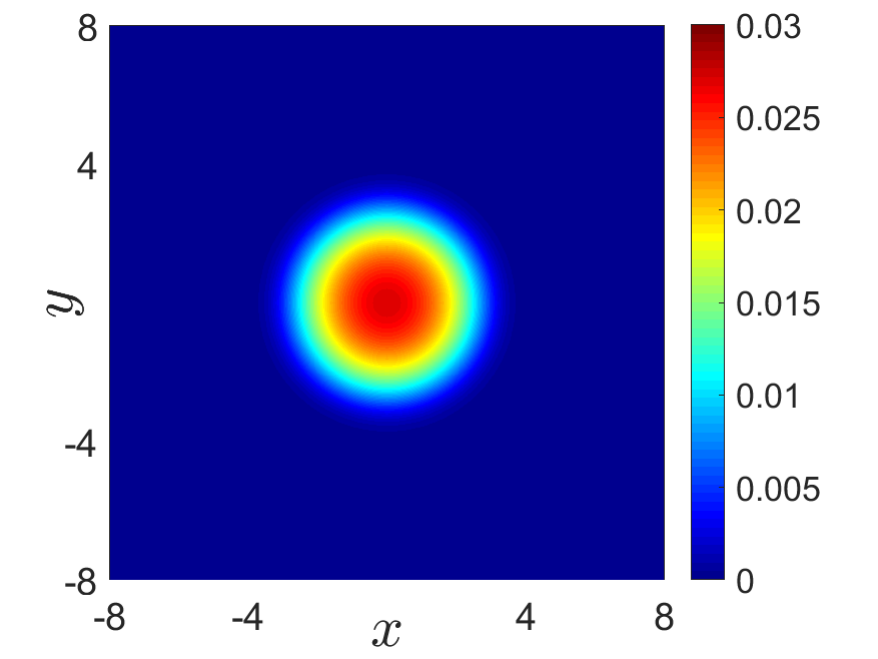}
  \hspace{-0.5cm}
  \includegraphics[scale=0.38]{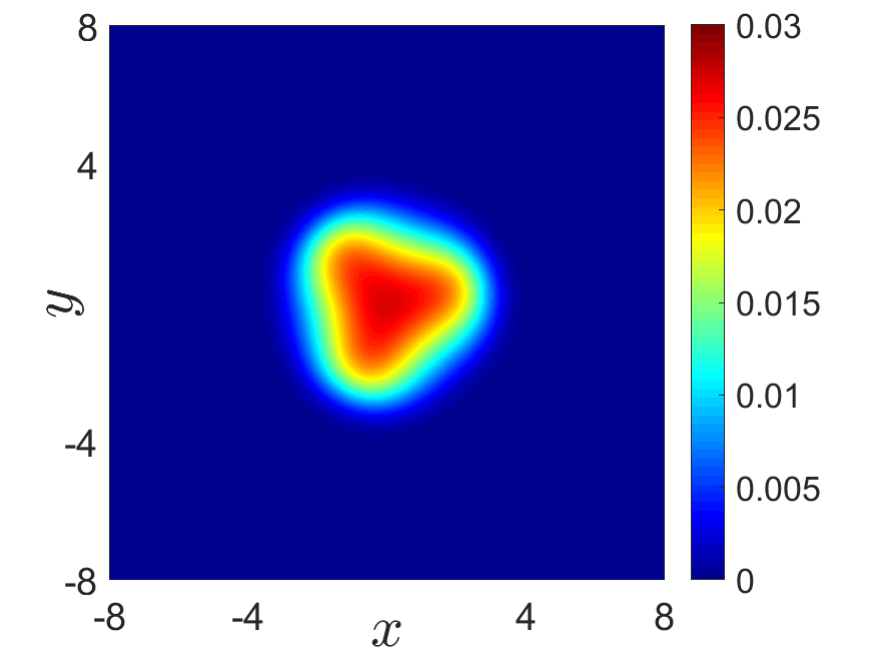}
  \hspace{-0.5cm}
  \includegraphics[scale=0.38]{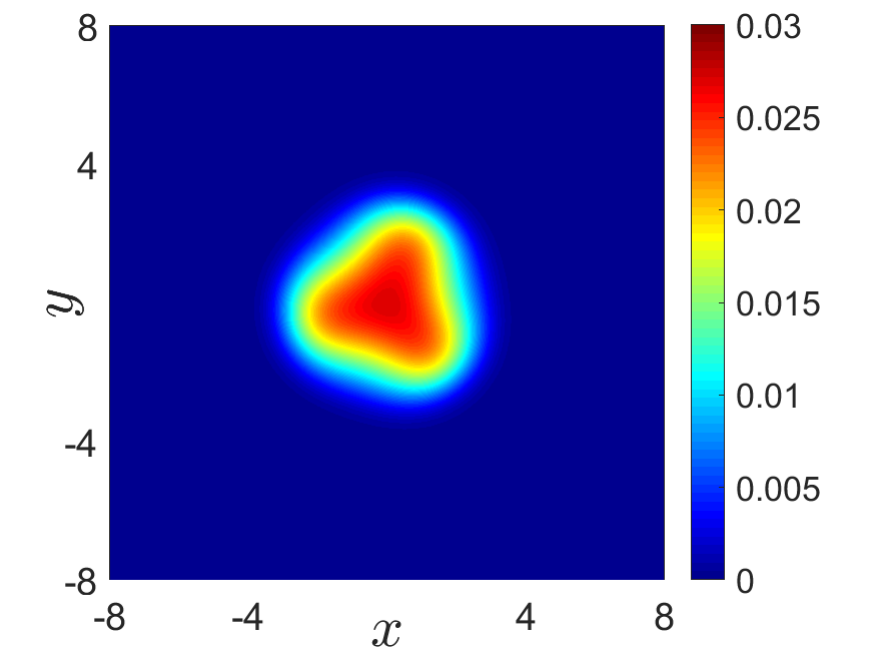}
 \\
  \includegraphics[scale=0.38]{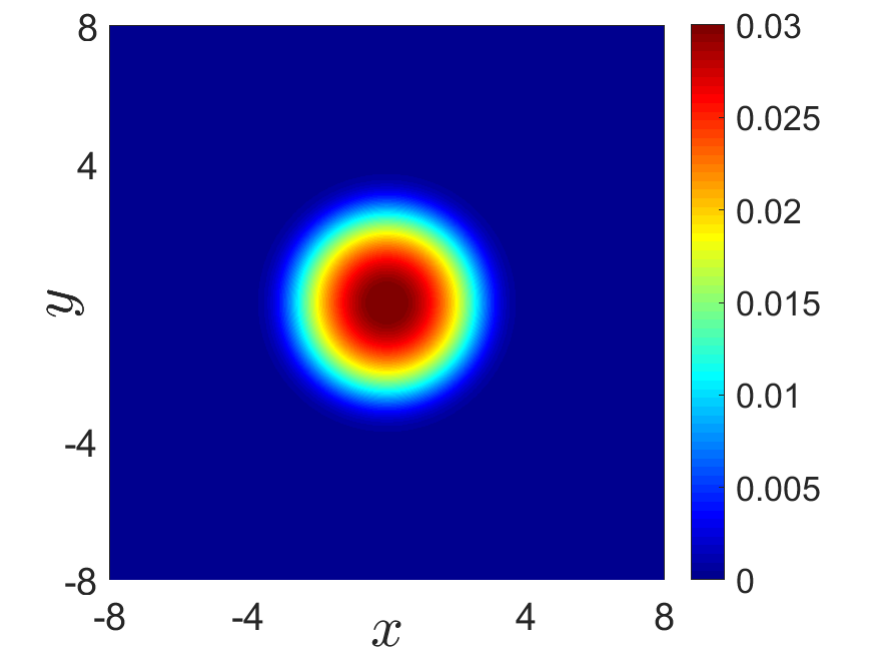}
  \hspace{-0.5cm}
  \includegraphics[scale=0.38]{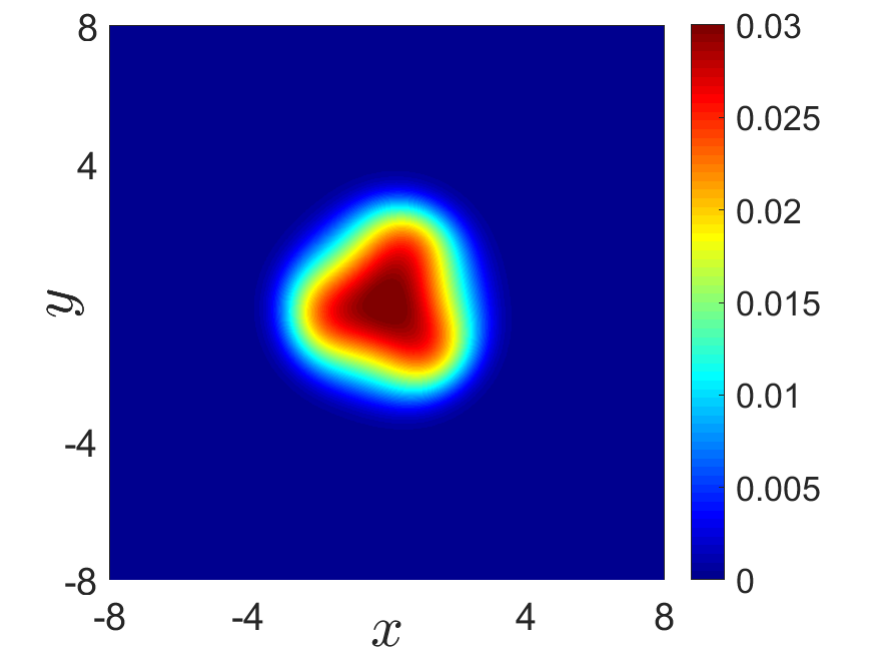}
  \hspace{-0.5cm}
  \includegraphics[scale=0.38]{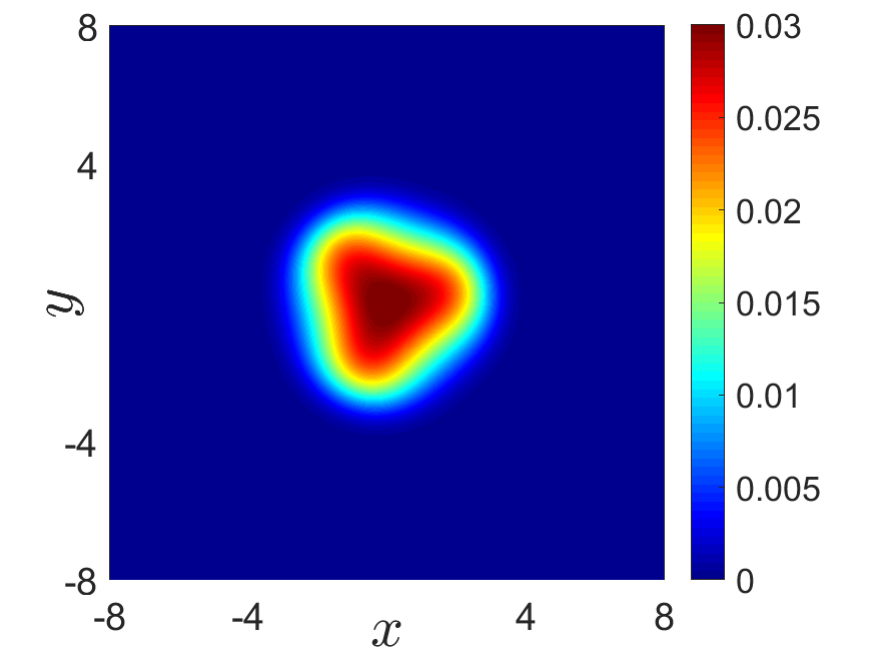}
\caption{Snapshots of the ground states $|\phi_j(\bx)|^2$ (left), perturbed densities $n_j^\ell(\bx,t=7.4)$ (middle) and $n_j^\ell(\bx,t=9.2)$ (right)
by excitation mode indexed $\ell=15$ for {\bf Case I} in Example \ref{perturb_2D_1} ($j=1$ (top), $2$ (bottom)).}
\label{ptb_iso_wJJ-11}
\end{figure}

\begin{figure}[!htp]
\centering
  \includegraphics[scale=0.38]{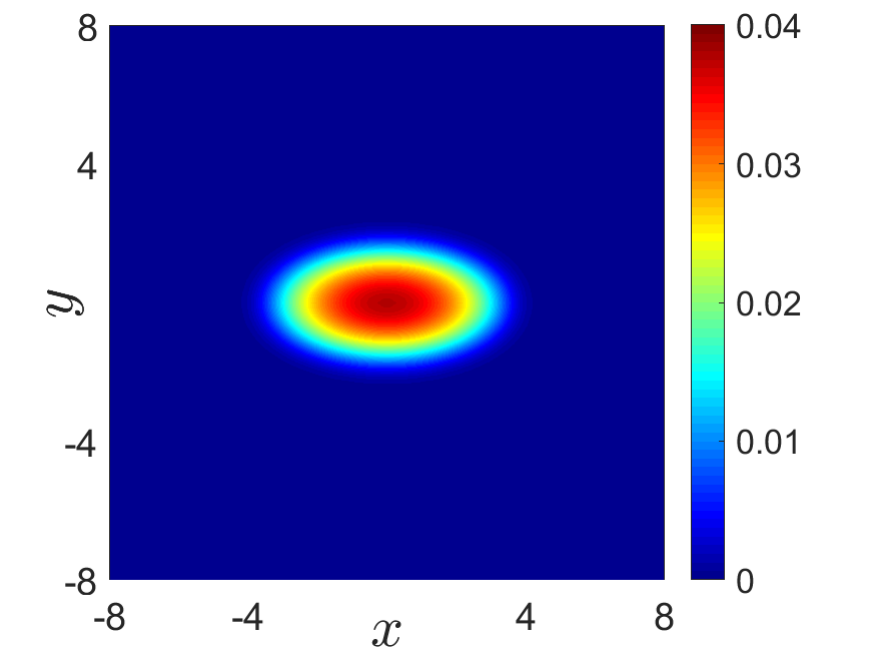}
  \hspace{-0.5cm}
  \includegraphics[scale=0.38]{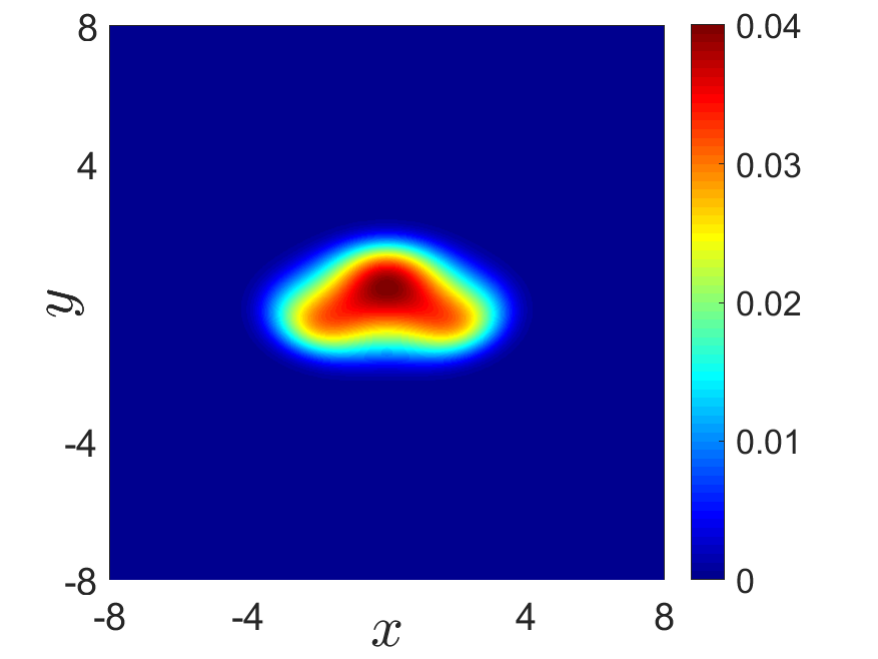}
  \hspace{-0.5cm}
  \includegraphics[scale=0.38]{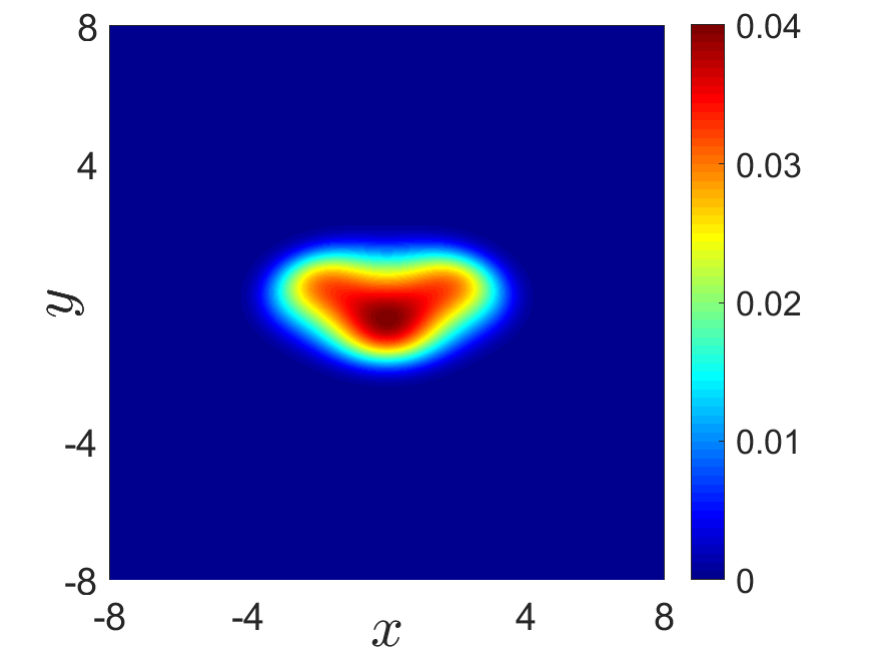}
\\
  \includegraphics[scale=0.38]{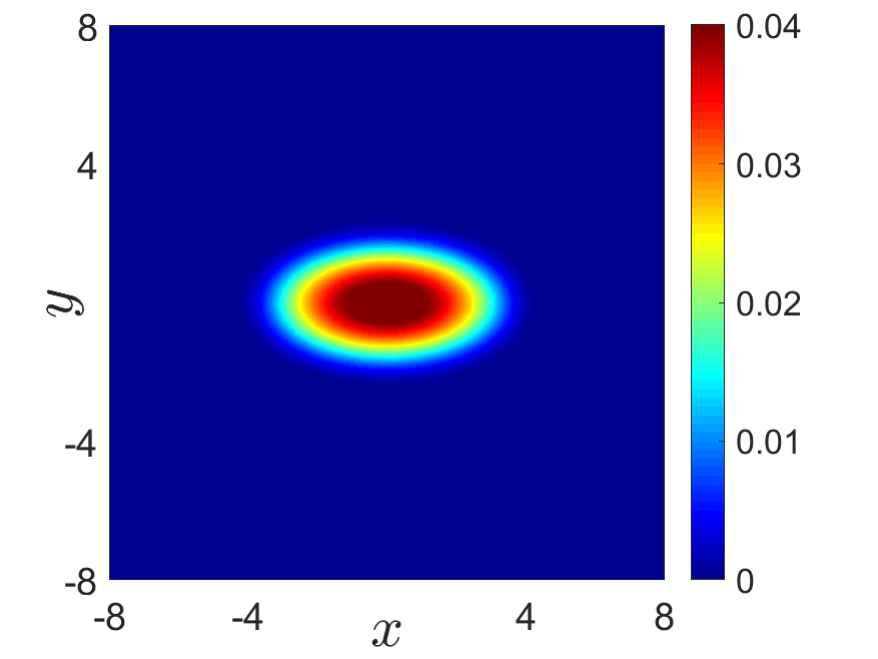}
  \hspace{-0.5cm}
  \includegraphics[scale=0.38]{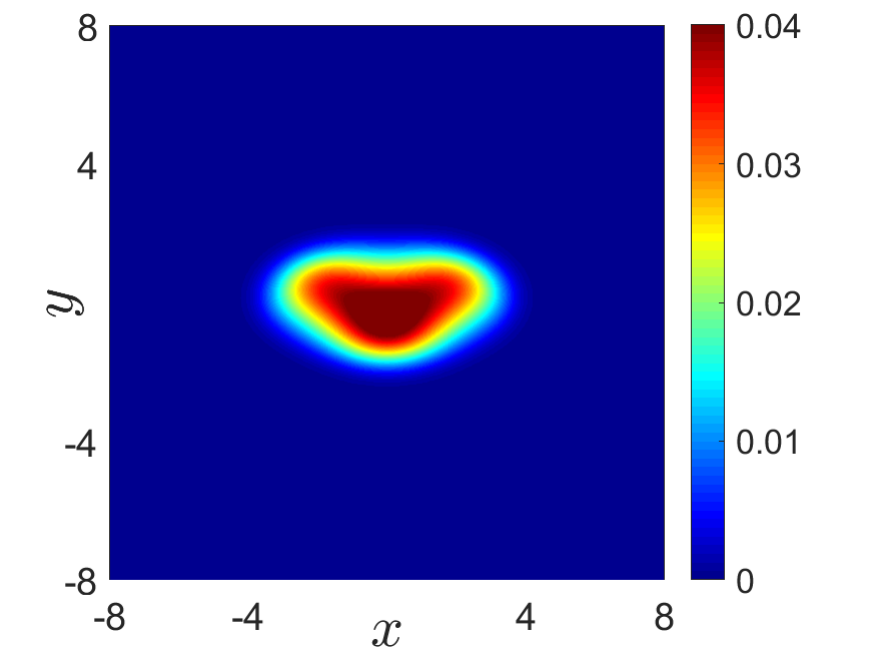}
  \hspace{-0.5cm}
  \includegraphics[scale=0.38]{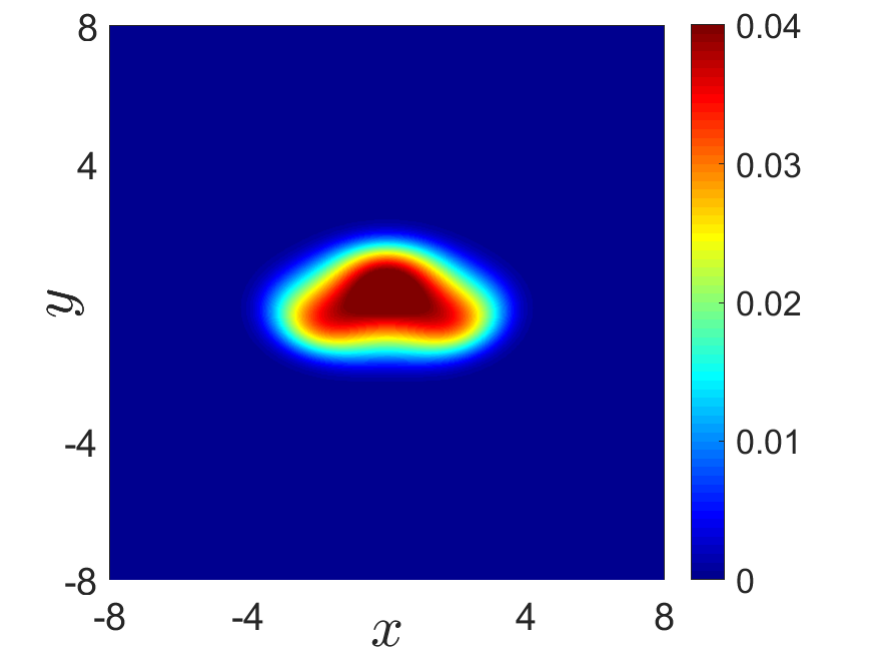}
\caption{Snapshots of the ground states $|\phi_j(\bx)|^2$ (left), perturbed densities $n_j^\ell(\bx,t=1.3)$ (middle) and $n_j^\ell(\bx,t=2.5)$ (right)
by excitation mode indexed $\ell=12$ for {\bf Case II} in Example \ref{perturb_2D_1} ($j=1$ (top), $2$ (bottom)).}
\label{ptb_iso_wJJ-12}
\end{figure}

\begin{figure}[!htp]
\centering
  \includegraphics[scale=0.38]{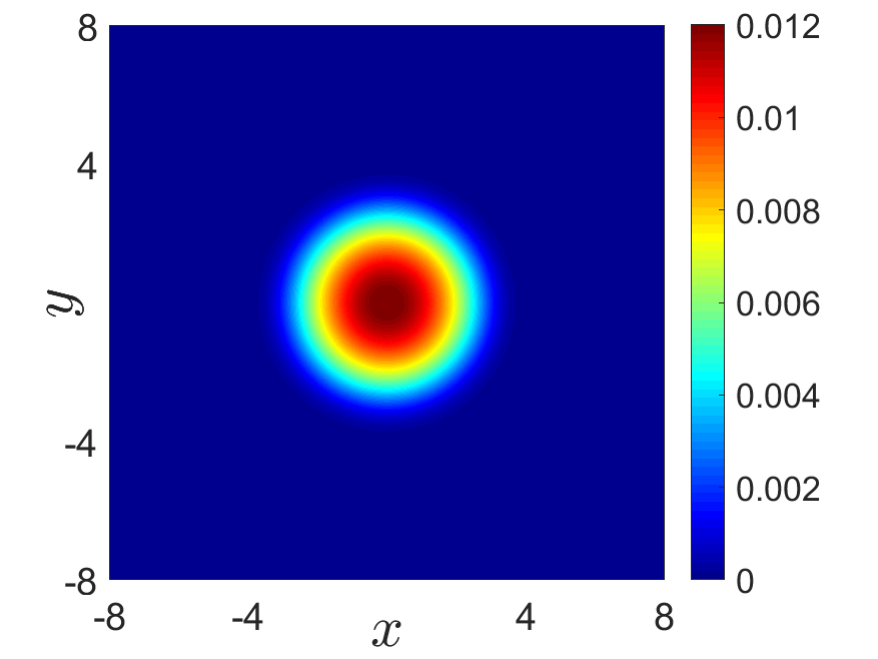}
  \hspace{-0.5cm}
  \includegraphics[scale=0.38]{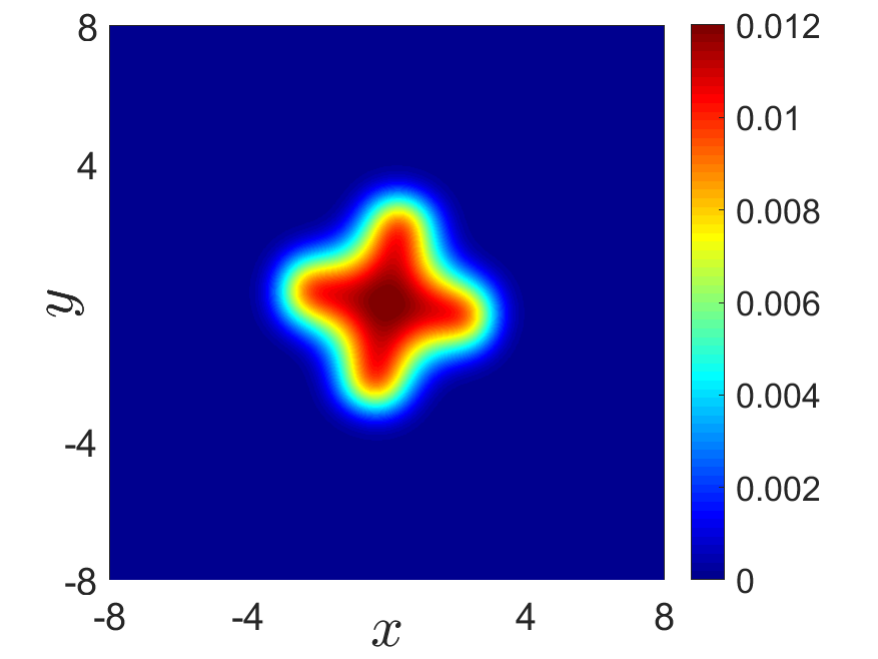}
  \hspace{-0.5cm}
  \includegraphics[scale=0.38]{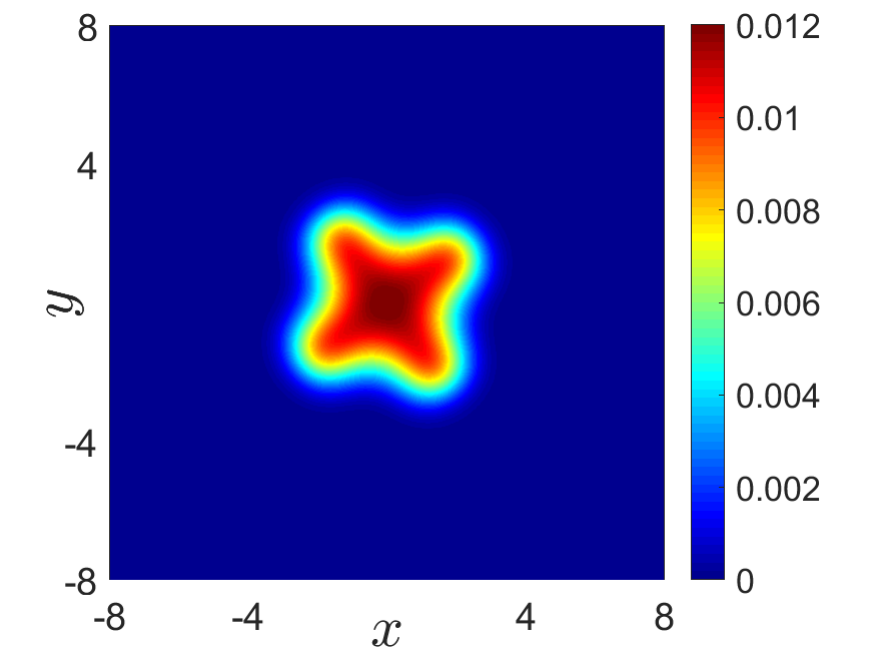}
 \\
  \includegraphics[scale=0.38]{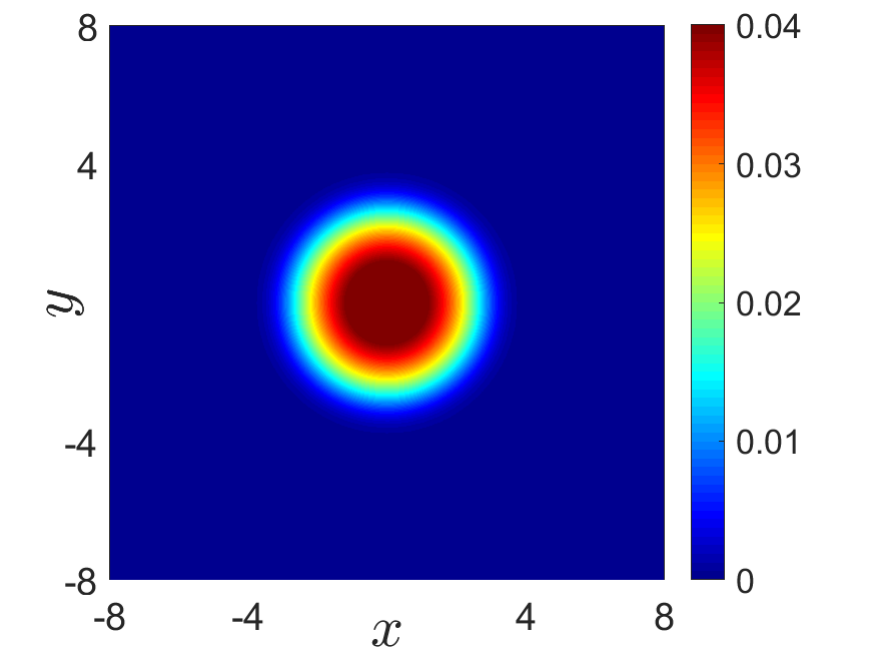}
  \hspace{-0.5cm}
  \includegraphics[scale=0.38]{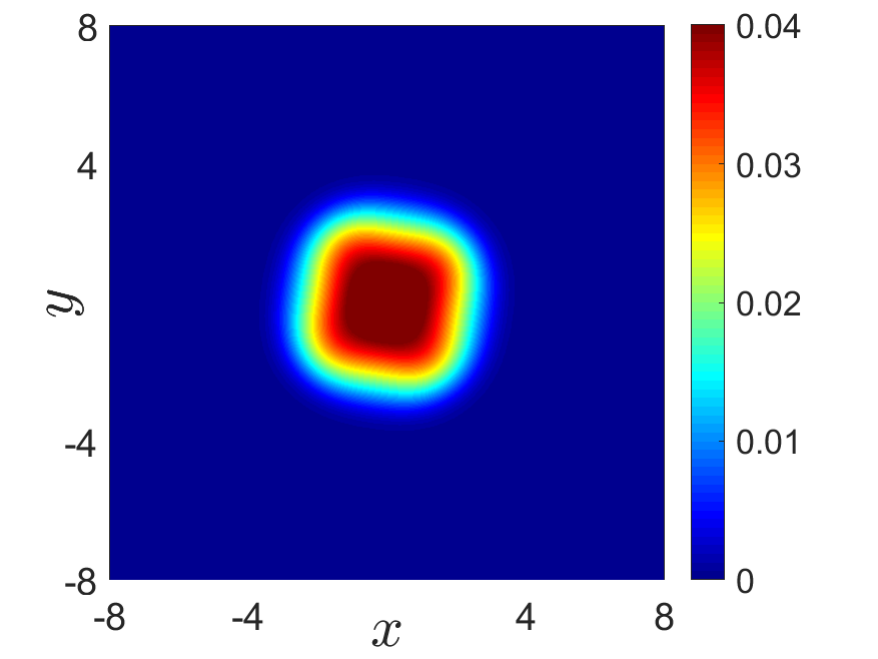}
  \hspace{-0.5cm}
  \includegraphics[scale=0.38]{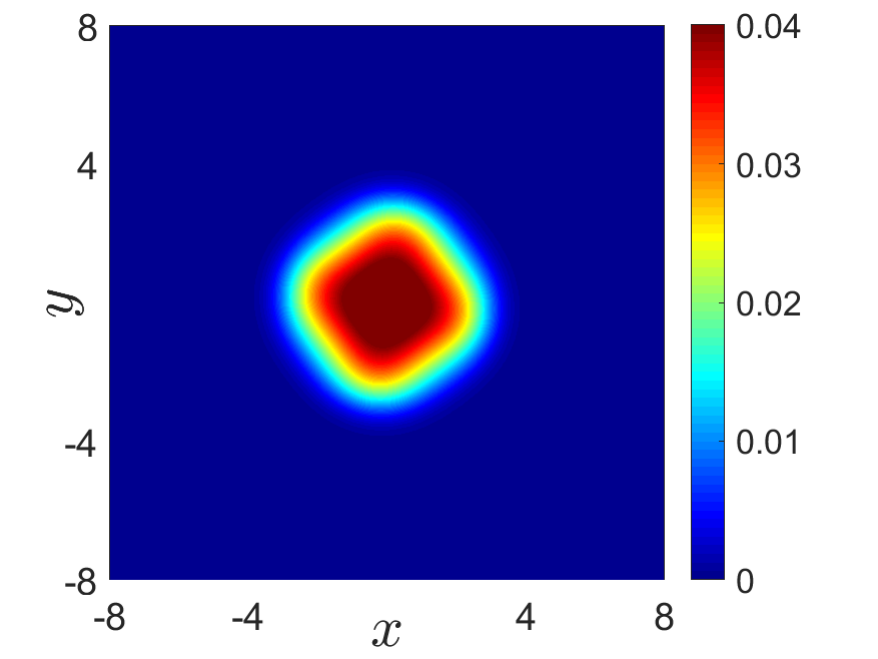}
\caption{
Snapshots of the ground states $|\phi_j(\bx)|^2$ (left), perturbed densities $n_j^\ell(\bx,t=2)$ (middle) and $n_j^\ell(\bx,t=4)$ (right)
by excitation mode indexed $\ell=14$ for {\bf Case III} in Example \ref{perturb_2D_1} ($j=1$ (top), $2$ (bottom)).}
\label{ptb_iso_woJJ-11}
\end{figure}

\begin{figure}[!htp]
\centering
  \includegraphics[scale=0.38]{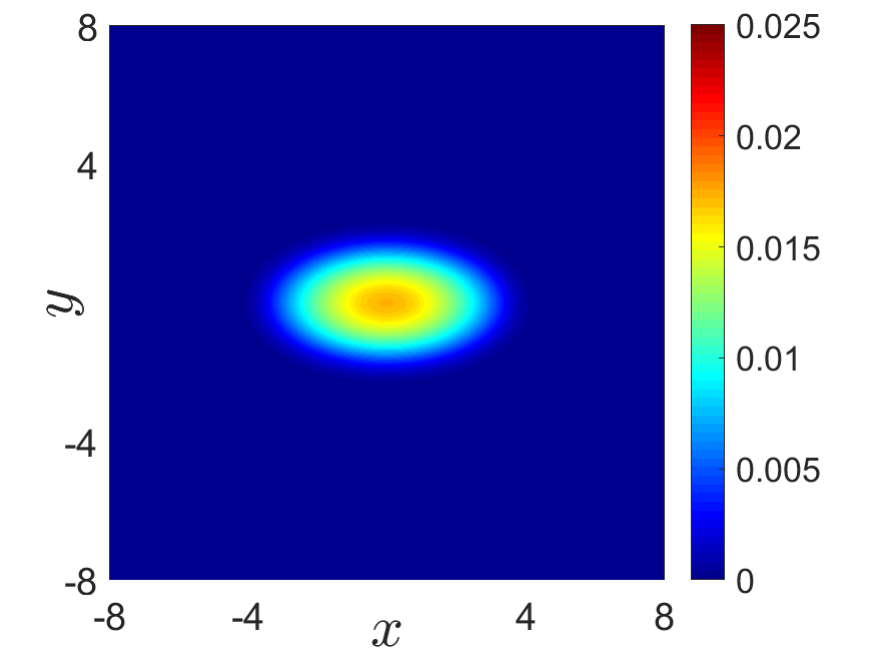}
  \hspace{-0.5cm}
  \includegraphics[scale=0.38]{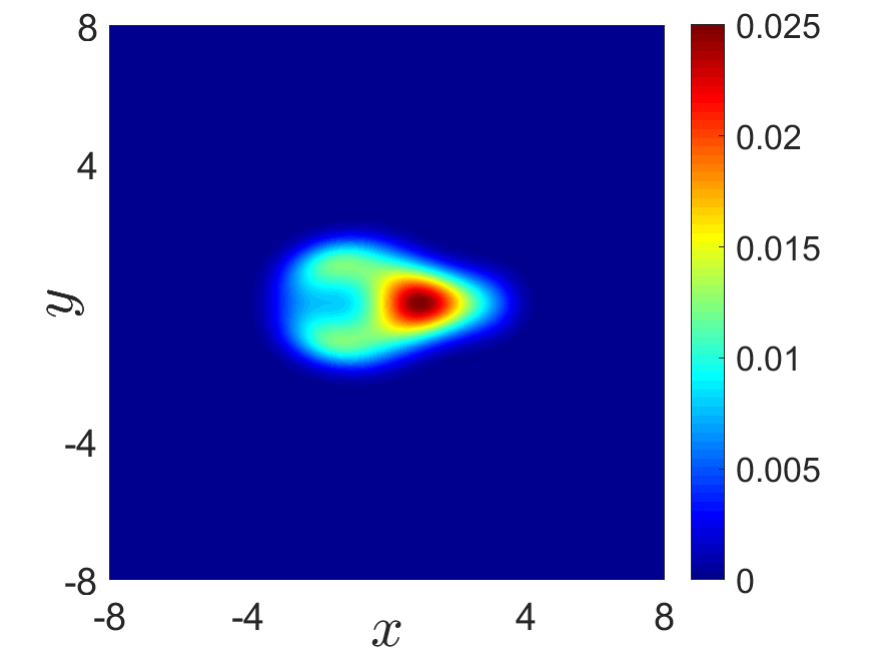}
  \hspace{-0.5cm}
  \includegraphics[scale=0.38]{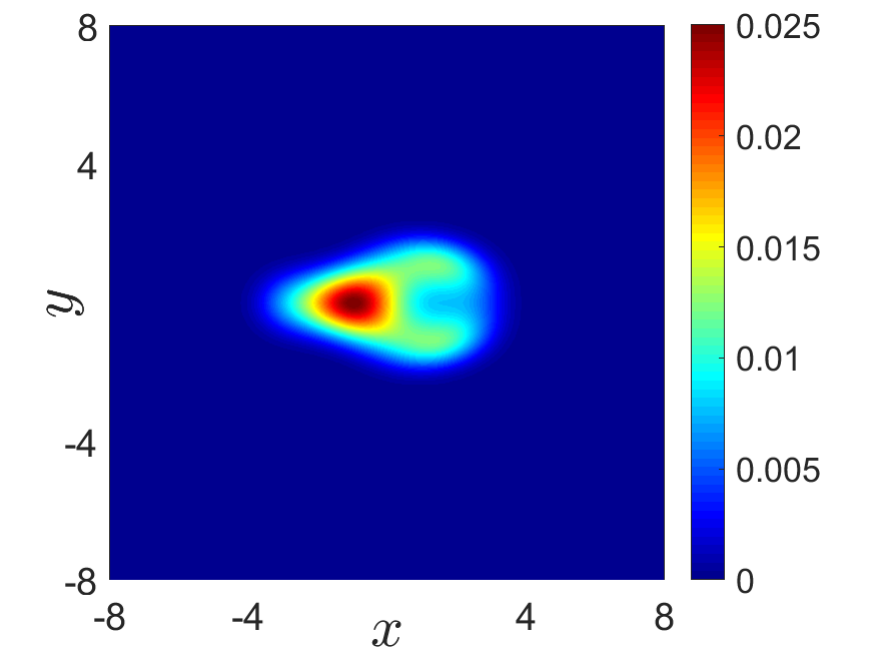}
\\
  \includegraphics[scale=0.38]{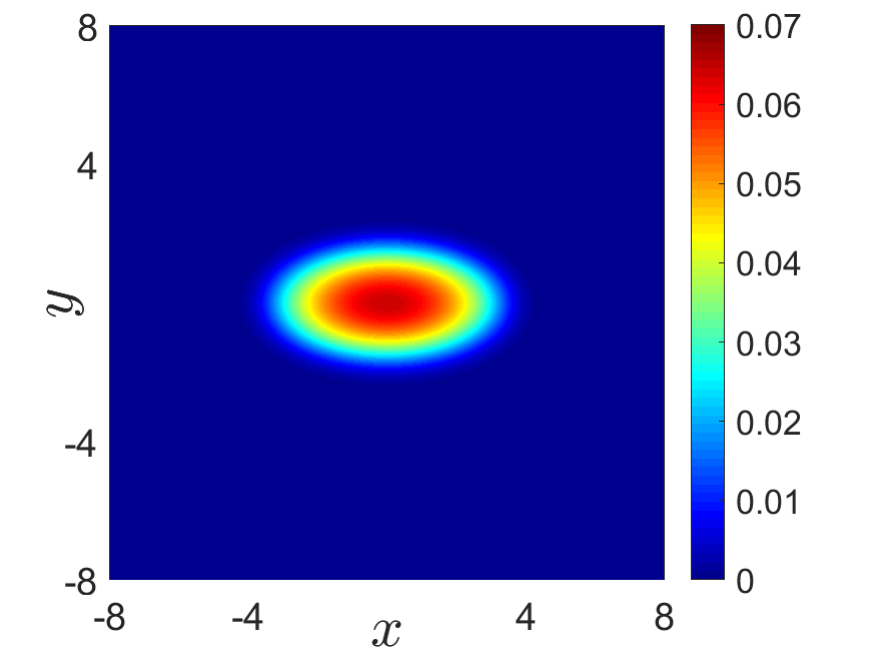}
  \hspace{-0.5cm}
  \includegraphics[scale=0.38]{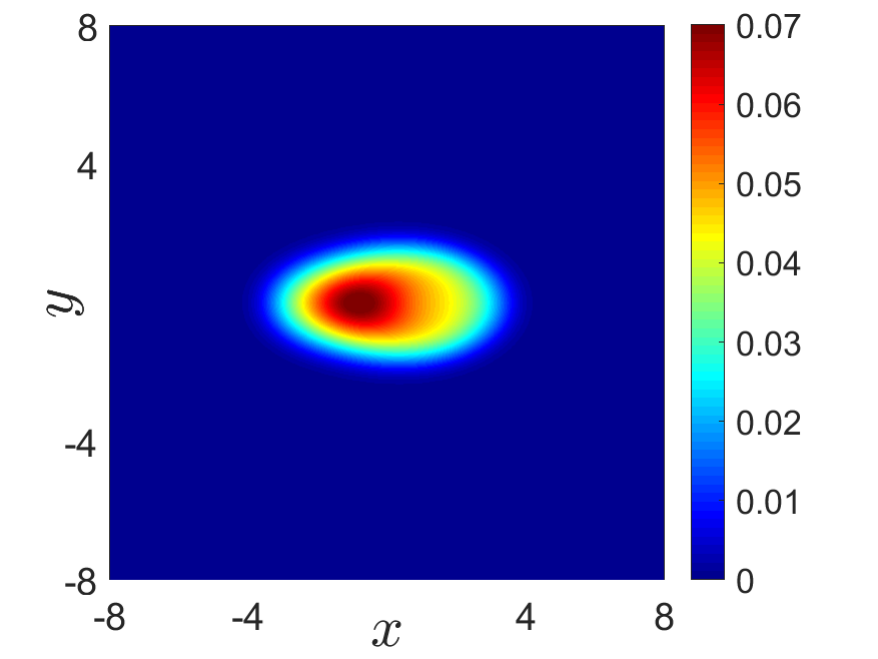}
  \hspace{-0.5cm}
  \includegraphics[scale=0.38]{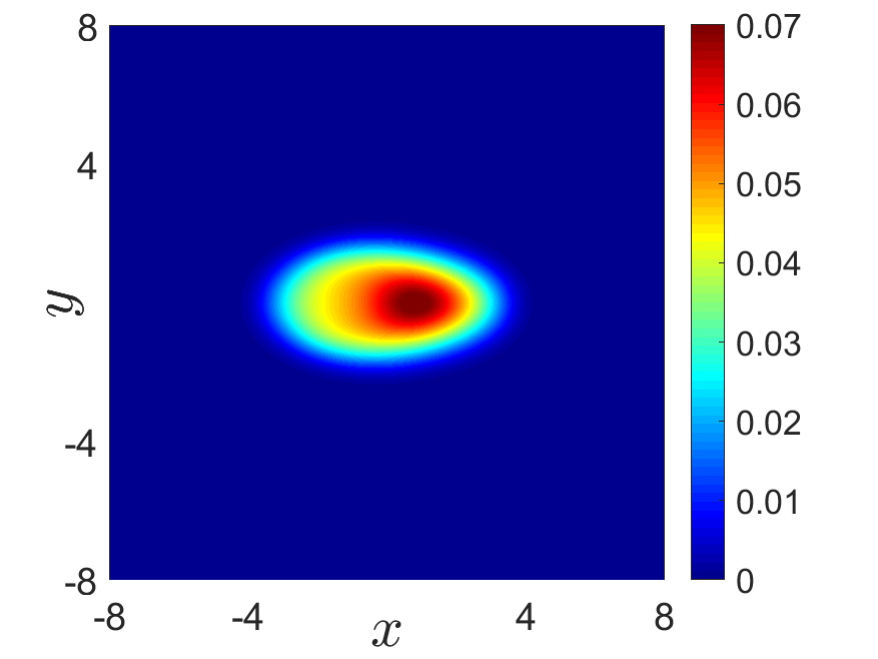}
\caption{
Snapshots of the ground states $|\phi_j(\bx)|^2$ (left), perturbed densities $n_j^\ell(\bx,t=5.4)$ (middle) and $n_j^\ell(\bx,t=6.8)$ (right)
by excitation mode indexed $\ell=16$ for {\bf Case IV} in Example \ref{perturb_2D_1} ($j=1$ (top), $2$ (bottom)).}
\label{ptb_iso_woJJ-12}
\end{figure}

Figure \ref{ptb_iso_wJJ-11}-\ref{ptb_iso_woJJ-12} display the perturbed density evolution $n_j^\ell(\bx,t)$ that are associated
with eigenvalues $\omega_\ell$ and Bogoliubov amplitudes $u^\ell_j,v^\ell_j$ for {\bf Case I--IV} in Example \ref{perturb_2D_1}, respectively. These figures show that the Josephson junction and external potential affect the shape of the perturbed density.
The perturbed densities are point/line-symmetric in a symmetric or antisymmetric external potential, respectively.
Meanwhile, the perturbed density will be compressed along the direction with a larger trapping frequency.
As depicted in the figures, the presence of Josephson junction brings in many more rich physical phenomena for the BdG equations,
which will be detailed in the future.

\subsection{Applications of 3D cases}

In this subsection, we investigate the excitation spectrum
and Bogoliubov amplitudes of the BdG equations of two-component condensates.
To visualize a normal mode, we show the isosurface plots of the Bogoliubov amplitudes.
The numerical computation is carried out with a very fine mesh $h=1/8$.

\begin{exam}
\label{amplitudes1}
We study the Bogoliubov amplitudes of the BdG equations of two-component condensates with/without a Josephson junction in 3D.
In this example, we study the following four cases:
\begin{itemize}
\item[] \hspace{-0.5cm}{\bf Case I.}  $~~\Omega=1$, $\gamma_x = \gamma_y=\gamma_z=1$, $\ell=12$;
\item[] \hspace{-0.5cm}{\bf Case II.}  $~\Omega=1$, $\gamma_x = \gamma_y=\gamma_z/2=1$, $\ell=12$;
\item[] \hspace{-0.5cm}{\bf Case III.}  $\Omega=0$, $\gamma_x = \gamma_y =\gamma_z=1$, $\ell=7$;
\item[] \hspace{-0.5cm}{\bf Case IV.} $\Omega=0$, $\gamma_x = \gamma_y=\gamma_z/2=1$, $\ell=7$.
\end{itemize}
\end{exam}

\begin{figure}[htbp]
    \centering
    \begin{subfigure}[b]{0.47\textwidth}
  \includegraphics[width=0.48\textwidth, height=0.17\textheight]{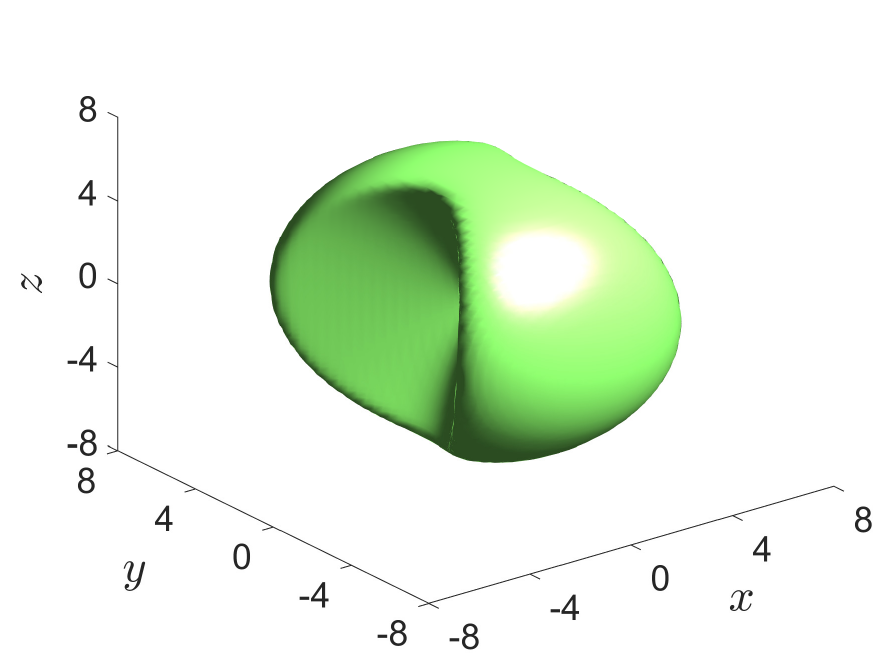}
  \includegraphics[width=0.48\textwidth, height=0.17\textheight]{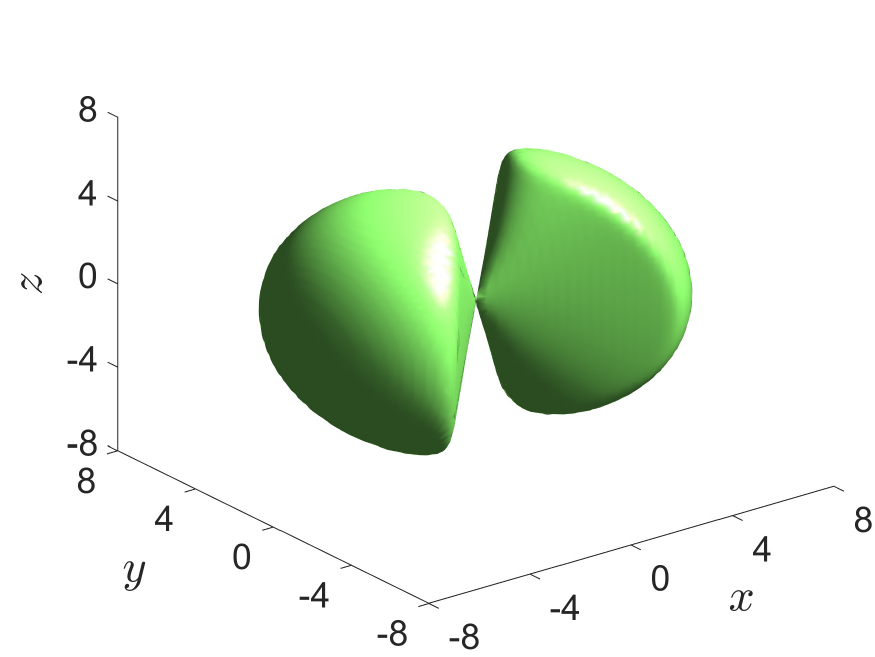}\\
  \includegraphics[width=0.48\textwidth, height=0.17\textheight]{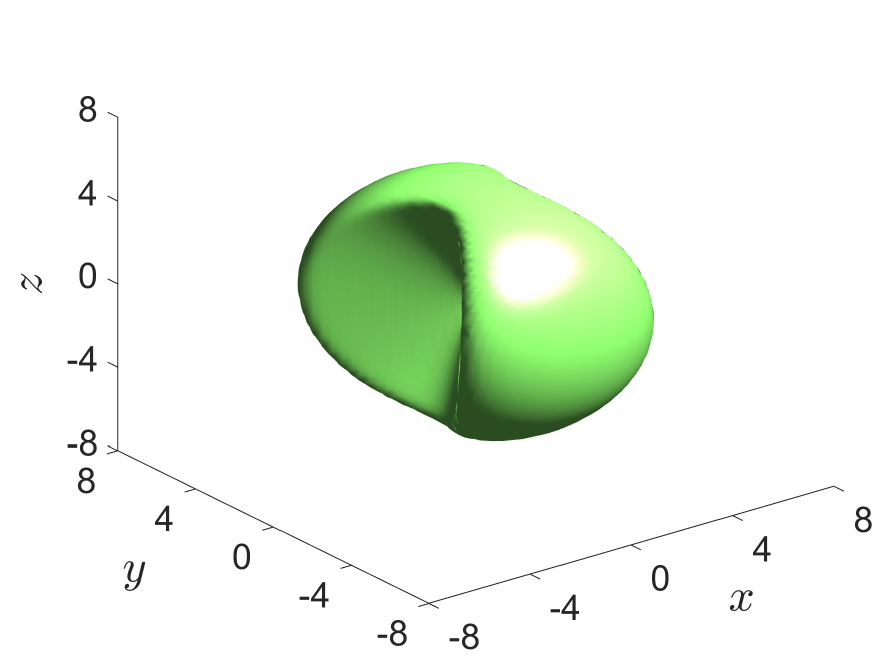}
  \includegraphics[width=0.48\textwidth, height=0.17\textheight]{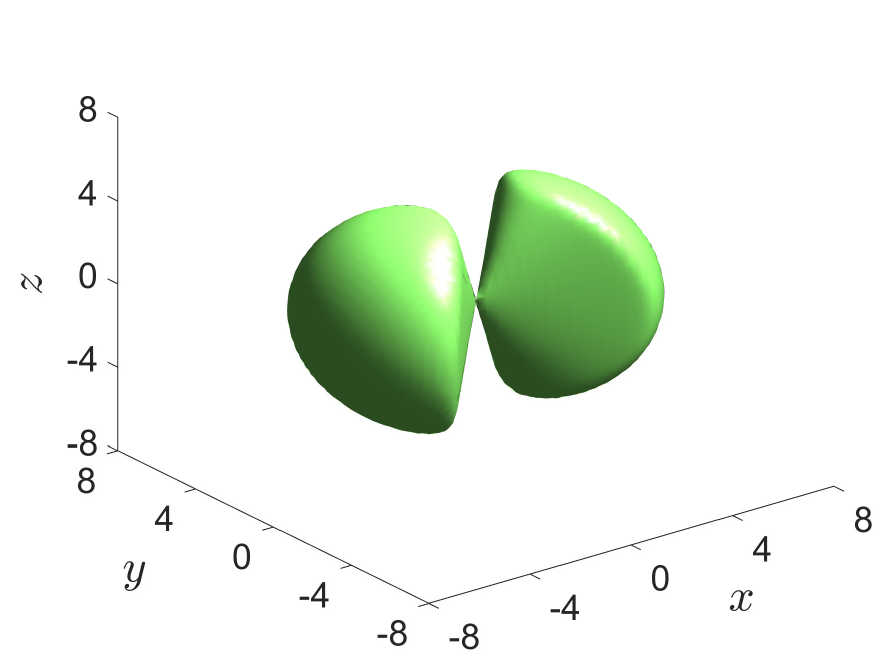}
  \caption{{\bf Case I}: $u_1^{12},u_2^{12}$ (top) and $v_1^{12},v_2^{12}$ (bottom).}
        \label{fig:sub1A1}
    \end{subfigure}
    \hfill
    \begin{subfigure}[b]{0.47\textwidth}
  \includegraphics[width=0.48\textwidth, height=0.17\textheight]{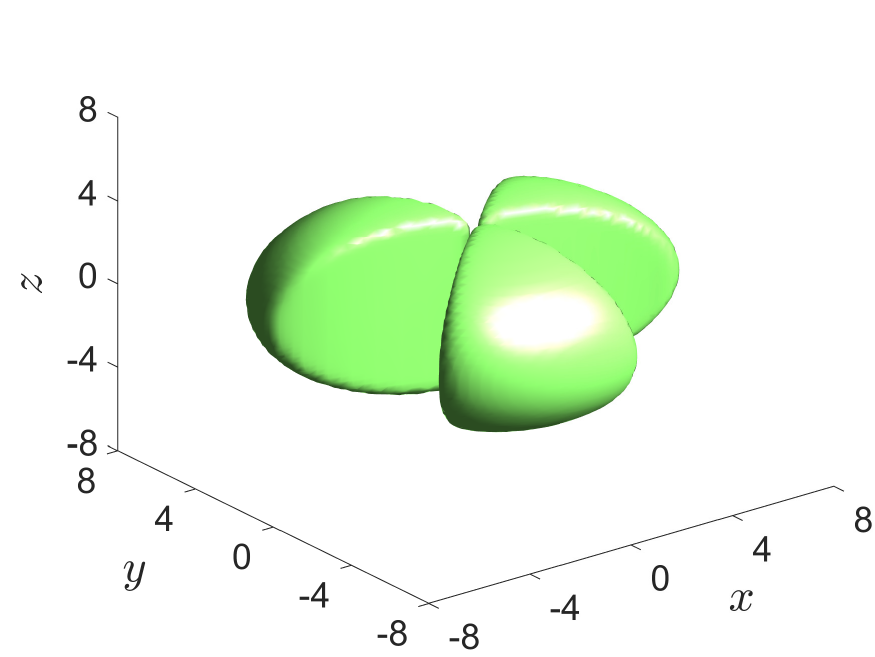}
  \includegraphics[width=0.48\textwidth, height=0.17\textheight]{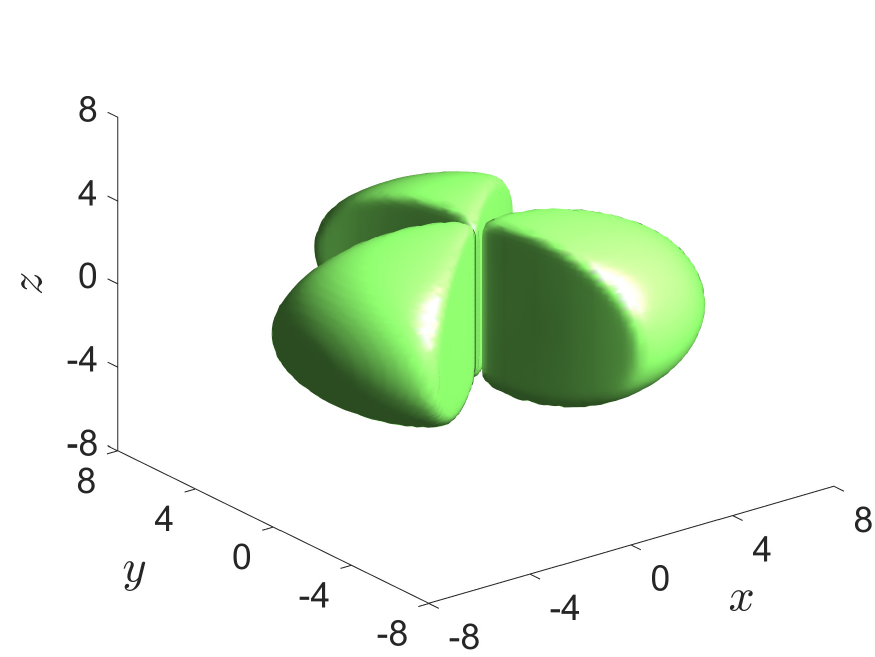}\\
  \includegraphics[width=0.48\textwidth, height=0.17\textheight]{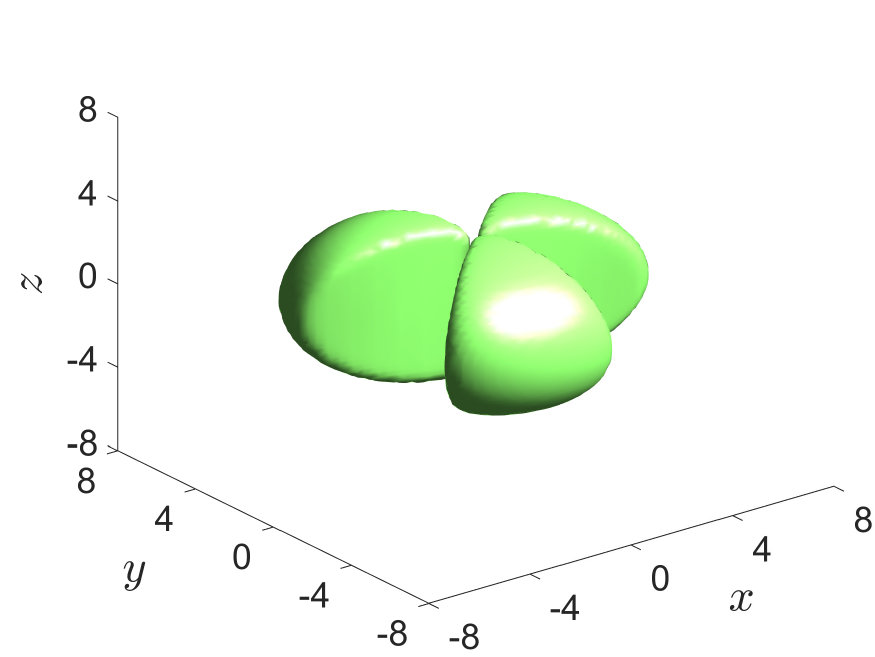}
  \includegraphics[width=0.48\textwidth, height=0.17\textheight]{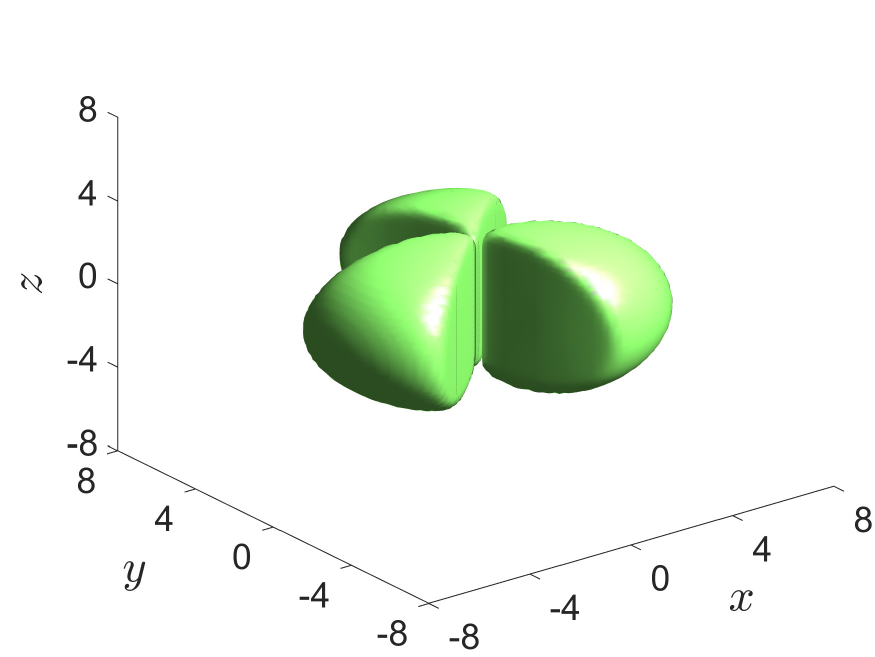}
        \caption{{\bf Case II}: $u_1^{12},u_2^{12}$ (top) and $v_1^{12},v_2^{12}$ (bottom).}
     \label{fig:sub1B1}
    \end{subfigure}
    \caption{Isosurface plots of the Bogoliubov amplitudes ($u_j^\ell=10^{-8}, v_j^\ell = 10^{-8}$) for {\bf Case I \& II}. }
    \label{eigv_iso_wJJ_111}
\end{figure}

\begin{figure}[htbp]
    \centering
    \begin{subfigure}[b]{0.47\textwidth}
  \includegraphics[width=0.48\textwidth, height=0.17\textheight]{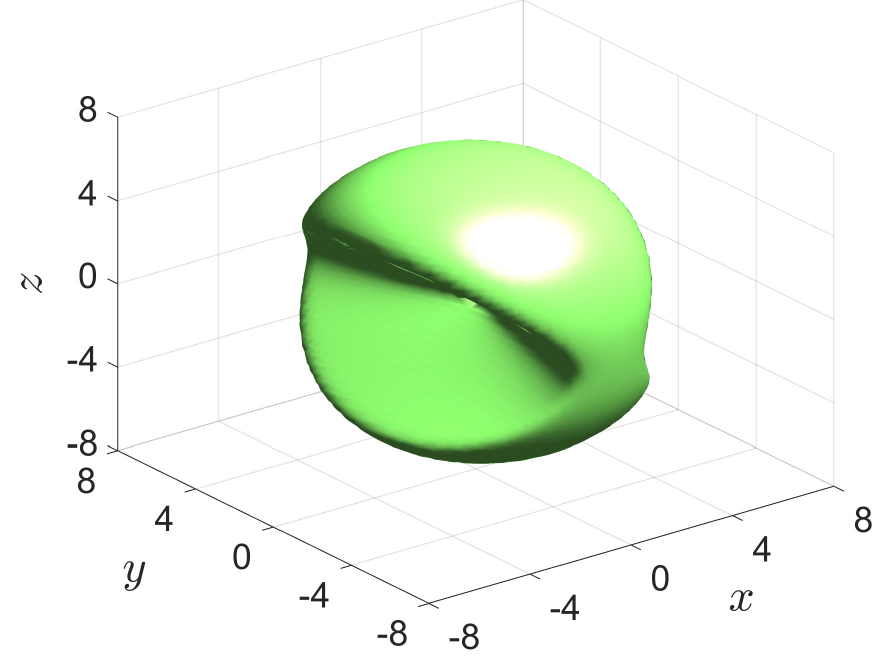}
  \includegraphics[width=0.48\textwidth, height=0.17\textheight]{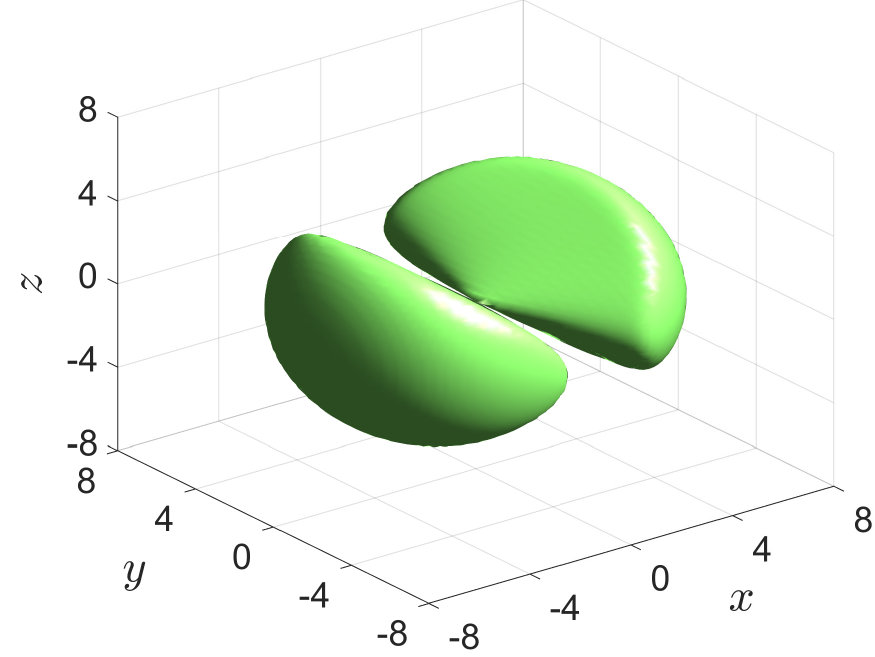} \\
  \includegraphics[width=0.48\textwidth, height=0.17\textheight]{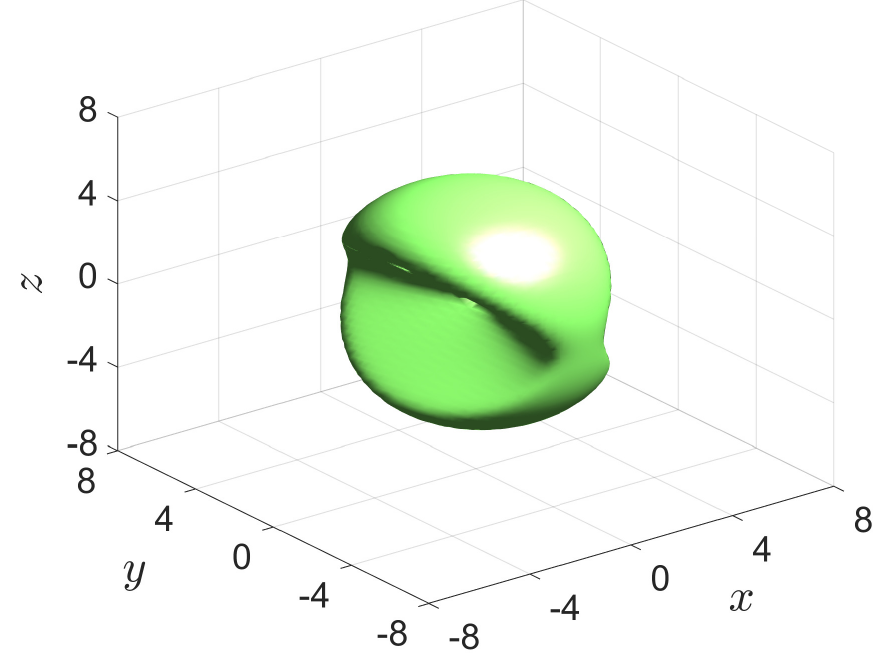}
  \includegraphics[width=0.48\textwidth, height=0.17\textheight]{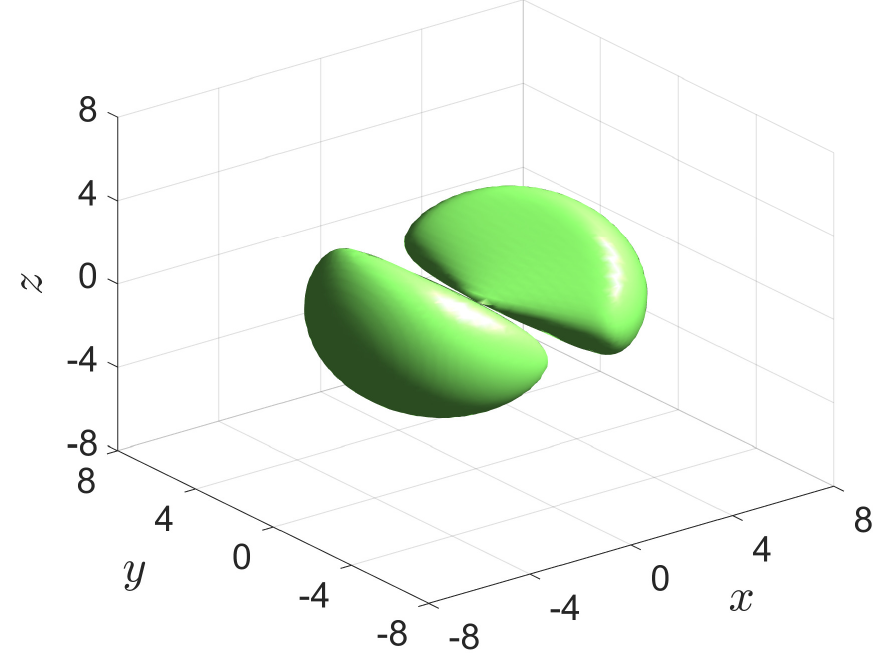}
  \caption{\textbf{Case III}: $u_1^{7},u_2^{7}$ (top) and $v_1^{7},v_2^{7}$ (bottom).}
        \label{fig:sub2A1}
    \end{subfigure}
    \hfill
    \begin{subfigure}[b]{0.47\textwidth}
  \includegraphics[width=0.48\textwidth, height=0.17\textheight]{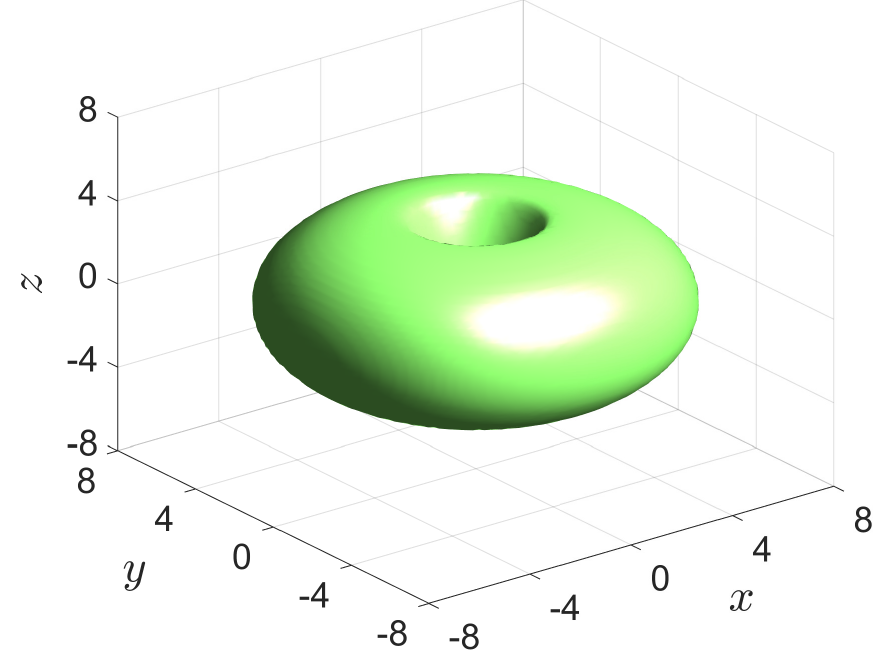}
  \includegraphics[width=0.48\textwidth, height=0.17\textheight]{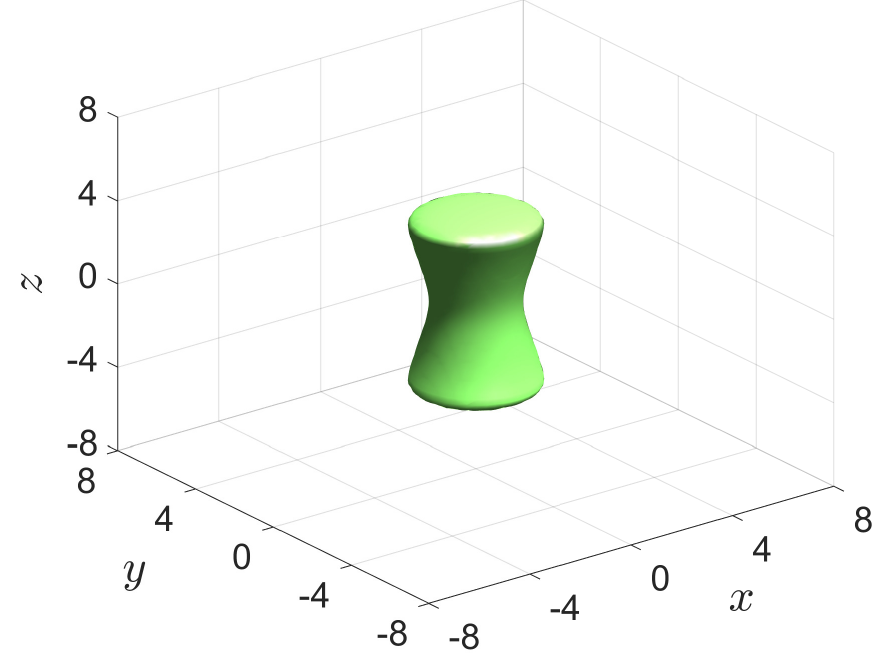}\\
  \includegraphics[width=0.48\textwidth, height=0.17\textheight]{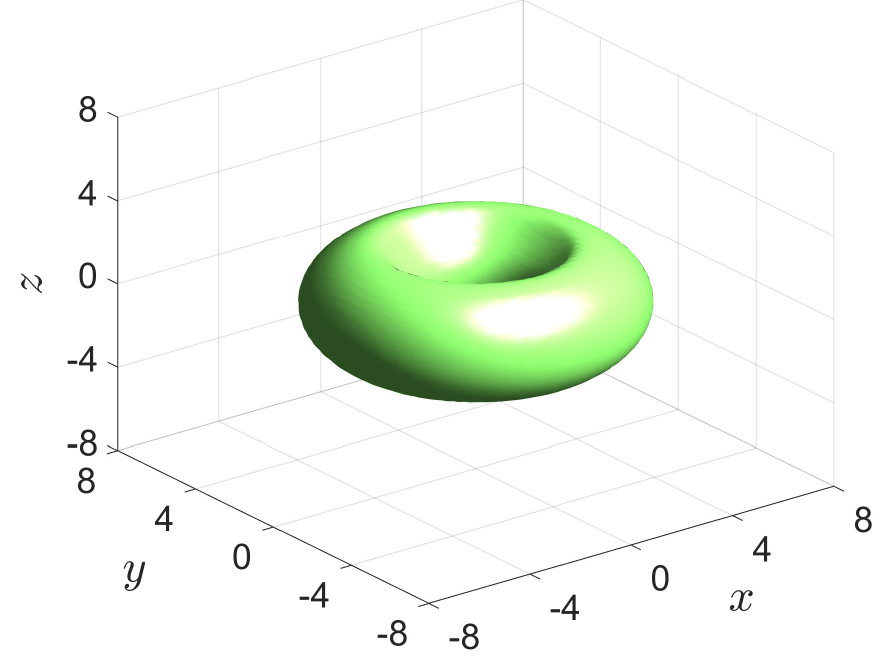}
  \includegraphics[width=0.48\textwidth, height=0.17\textheight]{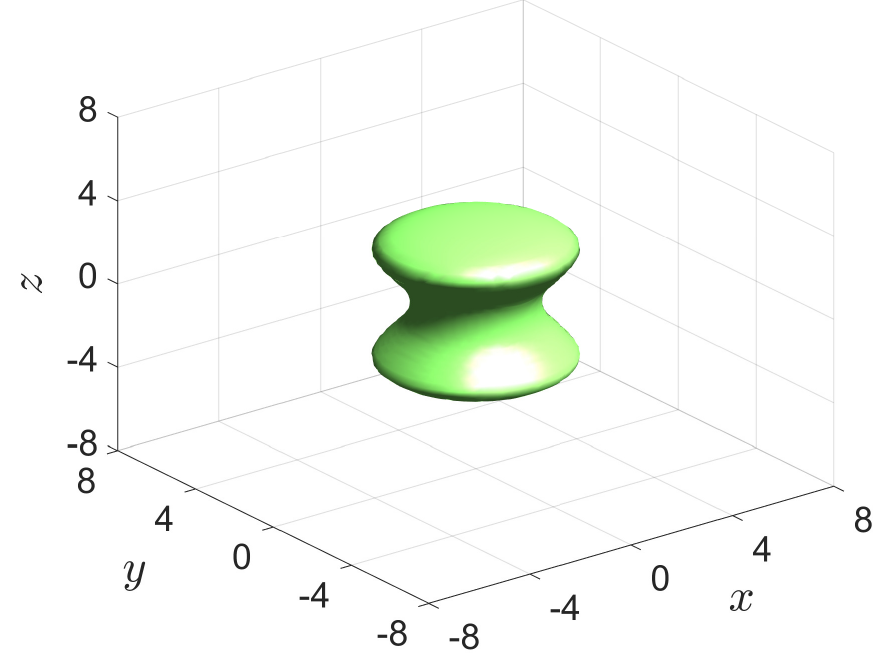}
        \caption{\textbf{Case IV}: $u_1^{7},u_2^{7}$ (top) and $v_1^{7},v_2^{7}$ (bottom).}
        \label{fig:sub2B1}
    \end{subfigure}
    \caption{Isosurface plots of the Bogoliubov amplitudes ($u_j^\ell=10^{-8}, v_j^\ell = 10^{-8}$) for {\bf Case III \& IV}. }
    \label{eigv_iso_woJJ_111}
\end{figure}

Figure \ref{eigv_iso_wJJ_111}-\ref{eigv_iso_woJJ_111} display isosurface plots of the eigenmodes
$\bu^\ell = (u^\ell_1,u^\ell_{2})^\top$ and $\bv^\ell = (v^\ell_1,v^\ell_{2})^\top$ associated with the different $\ell$ for {\bf Case I--II} and {\bf Case III--IV} in Example \ref{amplitudes1}, respectively.
From these figures, we can see that the external potential affects the shape of the eigenmodes $\bu^\ell$ and $\bv^\ell$ obviously.
It seems that the isosurface plots of the 1 and 2 components are complementary to each other.
Meanwhile, all eigenmodes will be compressed along the direction with a larger trapping frequency.
The presence of Josephson junction and isotropic/anisotropic external potential brings in many rich phase diagrams for eigenmodes,
and we shall leave such topic as a future study.

%\clearpage

%%=============================================================

\section{Conclusion}
\label{sec:Conclusion}

In this work, we introduce a novel and highly efficient numerical method for solving the Bogoliubov-de Gennes (BdG) equations governing two-component BECs.
We first explore its analytical properties,
including the exact eigenpairs, structure of generalized nullspace, and bi-orthogonality of eigenspaces at a continuous level.
Then, by combining the Fourier spectral method and the modified Gram-Schmidt bi-orthogonal algorithm,
we propose a structure-preserving eigenvalue method for the resulting large-scale dense non-Hermitian discrete eigenvalue problem.
Particularly, we incorporate the generalized nullspace to deal with the notorious slow-convergence or even divergence phenomena
caused by eigenvalue zero, and achieved better numerical stability, convergence and efficiency.
We design an interactive interface for users to provide matrix-vector multiplication (or the operator-function evaluation),
therefore, no explicit matrix storage is required anymore, and it allows for efficient computations of large-scale and dense problems.
Specifically, the operator-function evaluation is implemented with a near-optimal complexity ${\mathcal O}(N_{\rm t}\log(N_{\rm t}))$ thanks to FFT.
Extensive numerical results are presented to showcase its superiority in terms of accuracy and efficiency.
We then investigate the excitation spectrum (eigenvalues) and Bogoliubov amplitudes (eigenfunctions) around the ground state
under different setups.
It is worth pointing out that we can easily extend our solver to BdG excitations associated with different BECs,
for example, the multi-component and spinor-dipolar BECs, with merely minor adjustments to the specific matrix-vector multiplication.

%%=============================================================

\section*{Acknowledgements}
We would like to thank Dr. Yu Li for his help with BOSP implementation.
This work was partially supported by the National Natural Science Foundation of China No.12271400,
the National Key R\&D Program of  China No. 2024YFA1012803. This work was partially supported by
Yunnan Key Laboratory of Modern Analytical Mathematics and Applications No. 202302AN360007
while the second author (Y. Zhang) was visiting Yunnan Normal University in 2025.

%%=============================================================
\bibliographystyle{plain}

\end{document}

%% file: command.tex
\newcommand{\be}{\begin{equation}}
\newcommand{\ee}{\end{equation}}
\newcommand{\bes}{\begin{equation*}}
\newcommand{\ees}{\end{equation*}}
\newcommand{\ba}{\begin{array}}
\newcommand{\ea}{\end{array}}
\newcommand{\bea}{\begin{eqnarray}}
\newcommand{\eea}{\end{eqnarray}}
\newcommand{\beas}{\begin{eqnarray*}}
\newcommand{\eeas}{\end{eqnarray*}}

\newcommand{\fl}[2]{\frac{#1}{#2}}
\newcommand{\bx}{\bm{x}}

\newcommand{\im}{{\mathrm i}}

\newcommand{\p}{\partial}
\newcommand{\gm}{{\gamma}}

\newcommand{\phiu}{{\phi_1}}
\newcommand{\phid}{{\phi_2}}

\newcommand{\phiug}{{\phi_1^g}}
\newcommand{\phidg}{{\phi_2^g}}
\newcommand{\phiua}{{\phi_1^\alpha}}
\newcommand{\phida}{{\phi_2^\alpha}}

\newcommand{\psiu}{{\psi_1}}
\newcommand{\psid}{{\psi_2}}

\newcommand{\muu}{{\mu_1}}
\newcommand{\mud}{{\mu_2}}
\newcommand{\muua}{{\mu_1^\alpha}}
\newcommand{\muda}{{\mu_2^\alpha}}

\newcommand{\betauu}{\beta_{11}}
\newcommand{\betaud}{\beta_{12}}
\newcommand{\betadu}{\beta_{21}}
\newcommand{\betadd}{\beta_{22}}

\newcommand{\diag}{{\rm diag}}

\newcommand{\bu}{{\bm u}}
\newcommand{\bv}{{\bm v}}
\newcommand{\bbf}{{\bm f}}
\newcommand{\bbg}{{\bm g}}

\newcommand{\og}{{\omega}}
\newcommand{\spa}{{\rm span}}
\newcommand{\ns}{{\rm null}}
\newcommand{\cI}{{\mathcal I_N}}
\newcommand{\err}{{\texttt{E}}}

\newcommand{\uul}{{u_1^\ell}}
\newcommand{\udl}{{u_2^\ell}}
\newcommand{\vul}{{v_1^\ell}}
\newcommand{\vdl}{{v_2^\ell}}

\newcommand{\cA}{{\mathcal A}}
\newcommand{\cB}{{\mathcal B}}
\newcommand{\cC}{{\mathcal C}}
\newcommand{\cD}{{\mathcal D}}

\newcommand{\dif}{{\mathrm d}}

\newcommand{\cH}{{\mathcal H}}

\newcommand{\sH}{{\mathsf H}}

\newcommand{\bmat}{\begin{bmatrix}}
\newcommand{\emat}{\end{bmatrix}}

\newcommand{\lf}{\left(}
\newcommand{\rg}{\right)}

\newcommand{\lag}{\left\langle}
\newcommand{\rag}{\right\rangle}

\newtheorem{remark}{Remark}%[section]

\newtheorem{thm}{Theorem}
\newtheorem{lem}{Lemma}
\newtheorem{corollary}{Corollary}

%% file: bdG-2com.bbl
\begin{thebibliography}{10}

\bibitem{AshhabPRA}
{\sc S. Ashhab and C. Lobo},
{\em External Josephson effect in Bose-Einstein condensates with a spin degree of freedom},
Phys. Rev. A, {\bf 66} (2002), article 013609.

\bibitem{ATZ-CiCP-PCG}
{\sc X. Antoine, Q. Tang and Y. Zhang},
{\em A preconditioned conjugated gradient method for
computing ground states of rotating dipolar
Bose-Einstein condensates via kernel truncation method
for dipole-dipole interaction evaluation},
Commun. Comput. Phys.,{\bf 24}(4) (2018), 966--988.



\bibitem{Bai2012Minimization}
{\sc Z. Bai and R.C. Li},
{\em Minimization principles for the linear response eigenvalue problem I: Theory},
SIAM J. Matrix Anal. Appl., {\bf 33}(4) (2012), 1075--1100.

\bibitem{Bai2013Minimization}
{\sc Z. Bai and R.C. Li},
{\em Minimization principles for linear response eigenvalue problem II: Computation},
SIAM J. Matrix Anal. Appl., {\bf 34}(2) (2013), 392--416.


\bibitem{Baillie2017Collective}
{\sc D. Baillie, R.M. Wilson and P.B. Blakie},
{\em Collective excitations of self-bound droplets of a dipolar quantum fluid},
Phys. Rev. Lett., {\bf 119}(25) (2017), article 255302.

\bibitem{BaoCai2com}
{\sc W. Bao and Y. Cai},
{\em Ground states of two-component Bose-Einstein
condensates with an internal atomic Josephson junction},
East Asian J. Appl. Math., {\bf 1}(1) (2011), 49--81.

\bibitem{Bao2017Mathematical}
{\sc W. Bao and Y. Cai},
{\em Mathematical models and numerical methods for spinor Bose-Einstein condensates},
Commun. Comput. Phys., {\bf 24}(4) (2018), 899--965.

\bibitem{Bao2007mass}
{\sc W. Bao and H. Wang},
{\em A mass and magnetization conservative and energy-diminishing numerical method for computing ground state of spin-1 Bose-Einstein condensates},
SIAM J. Numer. Anal., {\bf 45}(5) (2007), 2177--2200.



\bibitem{Bloch2008Many}
{\sc I. Bloch, J. Dalibard and W. Zwerger},
{\em Many-body physics with ultracold gases},
Rev. Mod. Phys., {\bf 80}(3) (2008), 885--964.

\bibitem{Boccato2020Acta}
{\sc C. Boccato, C. Brennecke, S. Cenatiempo and B. Schlein},
{\em Bogoliubov theory in the Gross-Pitaevskii limit},
Acta Math., {\bf 222} (2019), 219--335.



\bibitem{Chen2017Collective}
{\sc L. Chen, H. Pu, Z.-Q. Yu and Y.Zhang},
{\em Collective excitation of a trapped Bose-Einstein condensate with spin-orbit coupling},
Phys. Rev. A, {\bf 95}(3) (2017), article 033616.

\bibitem{Danaila2016Vector}
{\sc I. Danaila, M.A. Khamehchi, V. Gokhroo, P. Engels and P.G. Kevrekidis},
{\em Vector dark-antidark solitary waves in multicomponent Bose-Einstein condensates},
Phys. Rev. A, {\bf 94}(5) (2016), article 053617.

\bibitem{Deng2020Spin}
{\sc S.-X. Deng, T. Shi and S. Yi},
{\em Spin excitations in dipolar spin-1 condensates},
Phys. Rev. A, {\bf 102}(1) (2020), article 013305.

\bibitem{Edwards1996Collective}
{\sc M. Edwards, P.A. Ruprecht, K. Burnett, R.J. Dodd and C.W. Clark},
{\em Collective excitations of atomic Bose-Einstein condensates},
Phys. Rev. Lett., {\bf 77}(9) (1996), 1671--1674.


\bibitem{Gao2020Numerical}
{\sc Y. Gao and Y. Cai},
{\em Numerical methods for Bogoliubov-de Gennes excitations of Bose-Einstein condensates},
J. Comput. Phys., {\bf 403} (2020), article 109058.

\bibitem{Hall1998Dynamics}
{\sc D.S. Hall, M.R. Matthews, J.R. Ensher, C.E. Wieman and E.A. Cornell},
{\em Dynamics of component separation in a binary mixture of Bose-Einstein condensates},
Phys. Rev. Lett., {\bf 81}(8) (1998), 1539--1542.


\bibitem{Ho1998Spinor1/2}
{\sc T. -L. Ho and V. B. Shenoy},
{\em Binary mixtures of Bose Condensates of Alkali atoms},
Phys. Rev. Lett., {\bf 77}(16) (1996), 3276--3279.

\bibitem{Hu2004Analytical}
{\sc B. Hu, G. Huang and Y. Ma},
{\em Analytical solutions of the Bogoliubov-de Gennes equations for excitations of a trapped Bose-Einstein-condensed gas},
Phys. Rev. A, {\bf 69}(6) (2004), article 063608.

\bibitem{Huhtamaeki2011Elementary}
{\sc J.A.M. Huhtam{\"a}ki and P. Kuopanportti},
{\em Elementary excitations in dipolar spin-1 Bose-Einstein condensates},
Phys. Rev. A, {\bf 84}(4) (2011), article 043638.


\bibitem{YiLowLying18}
{\sc L. Jia, A.-B. Wang and S. Yi},
{\em Low-lying excitations of vortex lattices in condensates with anisotropic dipole-dipole interaction},
Phys. Rev. A, {\bf 97} (2018), article 043614.

\bibitem{Jin1996Collective}
{\sc D.S. Jin, J.R.Ensher, M.R. Matthews, C.E. Wieman and E.A. Cornell},
{\em Collective excitations of a Bose-Einstein condensate in a dilute gas},
Phys. Rev. Lett., {\bf 77}(3) (1996), 420--423.

\bibitem{ARPACK}
{\sc R.B. Lehoucq, D.C. Sorensen and C. Yang},
{\em ARPACK Users'Guide: Solution of Large Scale Eigenvalue Problems with Implicitly Restarted Arnoldi Methods},
 SIAM, Philadelphia, (1998).

\bibitem{LLXZ-BdGSpin1}
{\sc Y. Li, Z. Li, M. Xie and Y. Zhang},
{\em An efficient Fourier spectral algorithm for the Bogoliubov-de Gennes excitation eigenvalue problem},
preprint.


\bibitem{LiWangZhang}
{\sc Y. Li, Z. Wang and Y. Zhang},
{\em Bi-Orthogonal Structure-Preserving eigensolver for large-scale linear response problem},
preprint.

\bibitem{Liu2021Normalized}
{\sc W. Liu and Y. Cai},
{\em Normalized gradient flow with Lagrange multiplier for computing ground states of Bose-Einstein condensates},
SIAM J. Sci. Comput., {\bf 43}(1) (2021), B219--B242.

\bibitem{LiuZhang_opt}
{\sc X. Liu, Q. Tang, S. Zhang and Y. Zhang},
{\em On optimal zero padding of kernel truncation method},
SIAM J. Sci. Comput., {\bf 46}(1) (2024), A23--A49.


\bibitem{Matthews1998Dynamical}
{\sc M.R. Matthews, D.S. Hall, D.S. Jin, J.R. Ensher, C.E.Wieman, E.A. Cornell, F. Dalfovo, C. Minniti and S. Stringari},
{\em Dynamical response of a Bose-Einstein condensate to a discontinuous change in internal state},
Phys. Rev. Lett., {\bf 81}(2) (1998), 243--247.


\bibitem{Myatt1997Spin1/2}
{\sc C. J. Myatt, E. A. Burt, R. W. Ghrist, E. A. Cornell and C. E. Wieman},
{\em Production of two overlapping Bose-Einstein condensates by sympathetic cooling},
Phys. Rev. Lett., {\bf 78} (1997), 586--589.


\bibitem{Alkai-Japan}
{\sc T. Ohmi and K. Machida},
{\em Bose-Einstein Condensation with internal degrees of freedom in alkali atom gases},
J. Phys. Soc. Jpn., {\bf 67} (1998), 1822--1825.

\bibitem{BdG-FreeFEM}
{\sc G. Sadaka, V. Kalt, I. Danaila and Fr\'ed\'eric Hecht},
{\em A finite element toolbox for the Bogoliubov-de Gennes stability analysis of Bose-Einstein condensates},
Comput. Phys. Commun., {\bf 294} (2024), article 108948.

\bibitem{Shao2018structure}
{\sc M. Shao, F. H. da Jornada, L. Lin, C. Yang, J. Deslippe and S. G. Louie},
{\em A structure preserving Lanczos algorithm for computing the optical
absorption spectrum},
SIAM J. Matrix Anal. Appl., {\bf 39}(2) (2018), 683--711.

\bibitem{SpectralBkShen}
{\sc J.~Shen, T.~Tang and  L.-L.~Wang},
{\em Spectral Methods: Algorithms, Analysis and Applications}, Springer, Berlin, (2011).


\bibitem{Stenger1998Spin}
{\sc J. Stenger, S. Inouye, D.M. Stamper-Kurn, H.-J. Miesner, A.P. Chikkatur and W. Ketterle},
{\em Spin domains in ground-state Bose-Einstein condensates},
Nature, {\bf 396}(6709) (1998), 345--348.

\bibitem{Stringari1996Collective}
{\sc S. Stringari},
{\em Collective excitations of a trapped Bose-condensed gas},
Phys. Rev. Lett., {\bf 77}(12) (1996), 2360--2363.

\bibitem{BdG2com}
{\sc T. \'{S}wisłocki, T. Karpiuk and M. Brewczyk}
{\em Elementary excitations of a two-component Fermi system using the atomic-orbital approach}
Phys. Rev. A, {\bf 77}(3) (2008), article 033603.


\bibitem{Tang2022Spectrally}
{\sc Q. Tang, M. Xie, Y. Zhang and Y. Zhang},
{\em A spectrally accurate numerical method for computing the Bogoliubov-de Gennes excitations of dipolar Bose-Einstein condensates},
SIAM J. Sci. Comput., {\bf 44}(1) (2022), B100--B121.

\bibitem{TianCai}
{\sc T. Tian, Y. Cai, X. Wu and Z. Wen},
{\em Ground states of spin-F Bose-Einstein condensates},
SIAM J. Sci. Comput., {\bf 42}(4) (2020), B983--B1013.

\bibitem{Wang-2com}
{\sc H. Wang},
{\em Numerical simulations on stationary states for rotating two-component Bose-Einstein condensate}s,
J. Sci. Comput., {\bf 38} (2009), 149--163.

\bibitem{WilliamsPRA}
{\sc J. Williams, R. Walser, J. Cooper, E. Cornell and M. Holland},
{\em Nonlinear Josephson-type oscillations of a driven two-component Bose-Einstein condensate},
Phys. Rev. A (R), {\bf 59} (1999), R31--R34

\bibitem{Zhang2003Mean}
{\sc W. Zhang, S. Yi and L. You},
{\em Mean field ground state of a spin-1 condensate in a magnetic field},
New J. Phys., {\bf 5}(1) (2003), article 77.


\bibitem{Zhang2021}
{\sc Y. Zhang, X. Liu and M. Xie},
{\em Efficient and accurate computation of the Bogoliubov-de Gennes excitations for the quasi-2D dipolar Bose-Einstein  Condensates},
East Asian J. Appl. Math., {\bf 11}(4) (2021), 686--707.

\end{thebibliography}
